\documentclass{amsart2020}[12pt]

\newtheorem{theorem}{Theorem}[section]
\newtheorem{lemma}[theorem]{Lemma}
\newtheorem{proposition}[theorem]{Proposition}
\newtheorem{corollary}[theorem]{Corollary}
\newtheorem{problem}[theorem]{Problem}

\theoremstyle{definition}
\newtheorem{definition}[theorem]{Definition}
\newtheorem{example}[theorem]{Example}
\newtheorem{remark}[theorem]{Remark}

\usepackage{amscd,amssymb}
\usepackage{graphicx}
\usepackage[all]{xy}

\begin{document}

\title[Algebras, Functions, Trees, and Integrals]
{Graded Algebras, Algebraic Functions,\\
Planar Trees, and Elliptic Integrals}
\author[Vesselin Drensky]{Vesselin Drensky}
\date{}
\address{Institute of Mathematics and Informatics,
Bulgarian Academy of Sciences,
Acad. G. Bonchev Str., Block 8,
1113 Sofia, Bulgaria}
\email{drensky@math.bas.bg}

\thanks
{Partially supported by Grant KP-06 N 32/1 of 07.12.2019
``Groups and Rings -- Theory and Applications'' of the Bulgarian National Science Fund.}

\subjclass[2020]{16-02; 16P90; 16R10; 16S15; 16W50; 05A15; 05C05; 05C30; 08B20; 13A50; 15A72; 17A50; 33B10; 33C75.}
\keywords{Graded algebras, monomial algebras, algebras with polynomial identity, Hilbert series, algebraic series, transcendental series,
free magma, free nonassociative algebras, planar trees, elliptic integrals.}

\maketitle
\begin{center}
{\it Dedicated to the 70-th anniversary of Antonio Giambruno,\\
a mathematician, person and friend.}
\end{center}

\begin{abstract}
This article surveys results on graded algebras and their Hilbert series.
We give simple constructions of finitely generated graded associative algebras $R$ with Hilbert series
$H(R,t)$ very close to an arbitrary power series $a(t)$ with exponentially bounded nonnegative integer coefficients.
Then we summarize some related facts on algebras with polynomial identity.
Further we discuss the problem how to find series $a(t)$ which are rational/algebraic/transcendental over ${\mathbb Q}(t)$.
Applying a classical result of Fatou we conclude that if a finitely generated graded algebra has a finite Gelfand-Kirillov dimension,
then its Hilbert series is either rational or transcendental. In particular the same dichotomy holds
for the Hilbert series of a finitely generated algebra with polynomial identity.
We show how to use planar rooted trees to produce algebraic power series.
Finally we survey some results on noncommutative invariant theory which show that we can obtain as Hilbert series various algebraic functions
and even elliptic integrals.
\end{abstract}

\section*{Introduction}

We consider algebras $R$ over a field $K$. Except Sections \ref{section planar trees} and \ref{section invariant theory}
all algebras are finitely generated and associative. The field $K$ is an arbitrary of any characteristic except in
Section \ref{section invariant theory} when it is of characteristic 0.

The purpose of this article is to survey some results, both old and recent, on graded algebras and their Hilbert series.
In Section \ref{section growth and Hilbert series} we discuss the growth of algebras, and graded algebras and their Hilbert series.
Then in Section \ref{section prescribed Hilbert series} we give constructions of graded algebras with prescribed Hilbert series.
Section \ref{section PI-algebras} is devoted to algebras with polynomial identities, or PI-algebras.
We survey some results concerning basic properties and the growth of such algebras.
Section \ref{section power series} deals with power series with nonnegative integer coefficients.
We consider methods to produce series which are transcendental over ${\mathbb Q}(t)$
and graded algebras with transcendental Hilbert series.
Combining a classical result of Fatou from 1906 with a theorem of Shirshov from 1957 we obtain immediately
that the Hilbert series of a finitely generated graded PI-algebra is either rational or transcendental.
We also survey some constructions of algebraic power series based on automata theory and theory of formal languages.
In the next Section \ref{section planar trees} we consider a method for construction of algebraic power series with
nonnegative integer coefficients. The main idea is to combine results on planar rooted trees with number of leaves
divisible by a given integer with the fact that submagmas of free $\Omega$-magmas are also free.
Finally, in Section \ref{section invariant theory} we use methods of noncommutative invariant theory to construct free graded algebras
(also nonassociative and not finitely generated) with Hilbert series which are either algebraic or transcendental.
In particular, we give simple examples of free nonassociative algebras with Hilbert series which are elliptic integrals.

If not explicitly stated, all power series in our exposition will have nonnegative integer coefficients.
Usually, when we state theorems about power series we do not present them in the most general form and restrict the considerations
to the case of nonnegative integer coefficients.

\section{Growth of algebras and Hilbert series}\label{section growth and Hilbert series}

If $R$ is a finite dimensional algebra we can measure how big it is using its dimension $\dim(R)$ as a vector space.
But how to measure infinite dimensional algebras? If $R$ is an algebra (not necessarily associative)
generated by a finite dimensional vector space $V$, then the {\it growth function}
of $R$ is defined by
\[
g_V(n)=\dim(R^n),\quad R^n=V^0+V^1+V^2+\cdots+V^n,\quad n=0,1,2,\ldots.
\]
This definition has the disadvantage that depends on the generating vector space $V$.
For example, the algebra of polynomials in $d$ variables $K[X_d]=K[x_1,\ldots,x_d]$ is generated by the vector space $V$
with basis $X_d=\{x_1,\ldots,x_d\}$ and the growth function $g_V(n)$ is
\[
g_V(n)=\binom{n+d}{d}=\frac{(n+d)(n+d-1)\cdots(n+1)}{d!}
=\frac{n^d}{d!}+{\mathcal O}(n^{d-1}).
\]
The algebra $K[X_d]$ is generated also by the monomials of first and second degree,
i.e. by the vector space $W=V+V^2$. Then
\[
g_W(n)=\binom{2n+d}{d}=\frac{2^dn^d}{d!}+{\mathcal O}(n^{d-1}).
\]
What is common between both generating functions? There is a standard method to compare
eventually monotone increasing and positive valued functions
$f:{\mathbb N}_0={\mathbb N}\cup\{0\}\to {\mathbb R}$.
This class of functions consists of all functions $f$ such that there exists an $n_0\in {\mathbb N}$
such that $f(n_0)\geq 0$ and $f(n_2)\geq f(n_1)\geq f(n_0)$
for all $n_2\geq n_1\geq n_0$.
We define a partial ordering $\preceq$ and equivalence $\sim$ on the set of such functions.
We assume that
$f_1\preceq f_2$ for two functions $f_1$ and $f_2$ (and $f_2$ {\it grows faster than} $f_1$) if and only if
there exist positive integers
$a$ and $p$ such that for all sufficiently large $n$
the inequality $f_1(n)\leq af_2(pn)$ holds
and $f_1\sim f_2$ if and only if $f_1\preceq f_2$ and $f_2\preceq f_1$.
This allows to obtain some invariant of the growth because
$g_V(n)\sim g_W(n)$ for any generating vector spaces $V$ and $W$ of the algebra $R$.
The equivalence is expressed in the following notion. The limit superior
\[
\text{GKdim}(R)=\limsup_{n\to\infty}\log_n(g_V(n))
\]
is called the {\it Gelfand-Kirillov dimension} of $R$.
It is known that $\text{GKdim}(R)$ does not depend on the system of generators of the algebra $R$.

Below we give a brief information for the values of the Gelfand-Kirillov dimension of finitely generated associative algebras.
For details we refer to the book by Krause and Lenagan \cite{KL}.

\begin{theorem}\label{values of GKdim}
{\rm (i)} If $R$ is commutative
then $\text{\rm GKdim}(R)$ is an integer equal to the transcendence degree of the algebra $R$.

{\rm (ii)} If $R$ is associative then $\text{\rm GKdim}(R)\in\{0,1\}\cup[2,\infty]$
and every of these reals is realized as a Gelfand-Kirillov dimension.
\end{theorem}

Part (i) of Theorem \ref{values of GKdim} is a classical result.
In part (ii) the restriction $\text{GKdim}(R)\not\in(1,2)$ is the {\it Bergman Gap Theorem} \cite{Bg}.
Algebras $R$ with $\text{GKdim}(R)\in[2,\infty)$ are realized by Borho and Kraft \cite{BK},
see also the modification of their construction in the book of the author \cite[Theorem 9.4.11]{D1}.
We shall mimic these constructions in the next section.

In the sequel we shall work with graded algebras.
The algebra $R$ is {\it graded} if it is a direct sum of vector subspaces $R_0,R_1,R_2,\ldots$
called {\it homogeneous components} of $R$
and
\[
R_mR_n\subset R_{m+n},\quad m,n=0,1,2,\ldots.
\]
It is convenient to assume that $R_0=0$ or $R_0=K$. In most of our considerations the generators of $R$ are of first degree.
The formal power series
\[
H(R,t)=\sum_{n\geq 0}\dim(R_n)t^n,
\]
is called the {\it Hilbert series} (or {\it Poincar\'e series}) of $R$.

We often shall work with power series with nonnegative integer coefficients
\[
a(t)=\sum_{n\geq 0}a_nt^n,\quad a_n\in{\mathbb N}_0.
\]
The advantage of studying such power series instead of the
sequence $a_n$, $n=0,1,2,\ldots$, of the coefficients of $a(t)$ is that we may apply the theory of analytic
functions or to find some recurrence relations. In particular, we
may find a closed formula for $a_n$ or to estimate its asymptotic behavior.

Our {\it rational functions} will be fractions of two polynomials with rational coefficients,
i.e. elements of the field ${\mathbb Q}(t)$.
Similarly, algebraic and transcendental functions are also over ${\mathbb Q}(t)$.
{\it Algebraic functions} $a(t)$ have the property that $p(t,a(t))=0$ for some polynomial $p(t,z)\in {\mathbb Q}[t,z]$
and {\it transcendental functions} do not satisfy any polynomial equation with rational coefficients.
As usually, if $a(t)$ converges in a neighborhood of $0$ to a rational, algebraic or transcendental function,
we say that $a(t)$ is also rational, algebraic or transcendental, respectively.

Algebraic functions have a nice characterization
given by  the Abel-Tannery-Cockle-Harley-Comtet theorem \cite[p. 287]{A}, \cite{Co, C1, C2, H, T} (see \cite{BD} for comments).

\begin{theorem}\label{Abel and Co}
The algebraic function
\[
f(t)=\sum_{n\geq 0}a_nt^n
\]
is $D$-finite, i.e. it satisfies a linear differential equation with coefficients
which are polynomials in $t$. Equivalently, its coefficients $a_n$ satisfy a linear recurrence
with coefficients which are polynomials in $n$.
\end{theorem}

We shall recall the usual definition of different kinds of growth of a sequence $a_n$, $n=0,1,2,\ldots$, of
complex numbers.
If there exist positive $b$ and $c$ such that
$\vert a_n\vert\leq bn^c$ for all $n$, we say that the sequence
is of {\it polynomial growth}. (We use this terminology
although it is more precise to say that the sequence
$a_n$, $n=0,1,2,\ldots$, is polynomially bounded.)
If there exist $b_1,b_2>0$ and $c_1,c_2>1$ such that $\vert a_n\vert\leq b_2c_2^n$ for all $n$
and $b_1c_1^{n_k}\leq \vert a_{n_k}\vert$ for a subsequence $a_{n_k}$, $k=0,1,2,\ldots$,
then the sequence is of {\it exponential growth}.
Finally, if for any $b,c>0$ there exists a subsequence
$a_{n_k}$, $k=0,1,2,\ldots$, such that $\vert a_{n_k}\vert > bn_k^c$
and for any $b_1>0$, $c_1>1$ the inequality
$\vert a_n\vert < b_1c_1^n$ holds for all sufficiently large $n$,
then the sequence is of {\it intermediate growth}.

The following statement is well known.

\begin{proposition}\label{algebraic Hilbert series}
The coefficients of an algebraic power series
\[
a(t)=\sum_{n\geq 0}a_nt^n
\]
are either of polynomial or of exponential growth.
\end{proposition}

Every algebra $R$ generated by a finite set $\{r_1,\ldots,r_d\}$
is a homomorphic image of the free associative algebra
$K\langle X_d\rangle=K\langle x_1,\ldots,x_d\rangle$. The map $\pi_0:x_i\to r_i$, $i=1,\ldots,d$,
is extended to a homomorphism $\pi:K\langle X_d\rangle\to R$ and $R\cong K\langle X_d\rangle/I$, $I=\ker(\pi)$.
If the ideal $I$ of $K\langle X_d\rangle$ is finitely generated, then the algebra $R$ is {\it finitely presented}.
An important special case of graded algebras is the class of {\it monomial algebras}.
Monomial algebras are factor algebras of $K\langle X_d\rangle$ modulo an ideal generated by monomials,
i.e. by elements of the free unitary semigroup $\langle X_d\rangle$.

Below we give some properties of Hilbert series. We start with commutative graded algebras.

\begin{theorem}\label{Hilbert series of commutative algebras}
Let $R$ be a finitely generated graded commutative algebra. Then:

{\rm (i)} {\rm (Theorem of Hilbert-Serre)} The Hilbert series $H(R,t)$ is a rational function with denominator which is a product of binomials
$1-t^m$.

{\rm (ii)} If
\[
H(R,t)=p(t)\prod\frac{1}{(1-t^{m_i})^{a_i}},\quad a_i\geq 1,\quad p(t)\in {\mathbb Z}[t],
\]
then the Gelfand-Kirillov dimension $\text{\rm GKdim}(R)$ is equal to the multiplicity of $1$ as a pole of $H(R,t)$:
If $p(1)\not=0$, then $\text{\rm GKdim}(R)=\sum a_i$.
\end{theorem}

The coefficients of the Hilbert series of a finitely generated commutative algebras
are a subject of many additional restrictions, see Macaulay \cite{M}.
The picture for noncommutative graded algebras is more complicated than in the commutative case.
Govorov \cite{G1} proved that if the set of monomials $U$ is finite, then
the Hilbert series of the monomial algebra $R=K\langle X\rangle/(U)$ is a rational function.
He conjectured \cite{G1, G2} that the same holds for the Hilbert series of finitely presented graded algebras.
By a theorem of Backelin \cite{Ba} this holds when the ideal $(U)$ is generated by a single homogeneous polynomial.
On the other hand Shearer \cite{Sh} presented an example of a finitely presented graded algebra with algebraic nonrational Hilbert series.
As he mentioned his construction gives also an example with a transcendental Hilbert series.
Another simple example of a finitely presented algebra with algebraic Hilbert series was given by Kobayashi \cite{K}.
It is interesting to mention that the rationality of the Hilbert series may depend on the base field $K$.
The following theorem is from the recent paper by Piontkovski \cite{Pi}.

\begin{theorem}\label{linear growth}
Let $K$ be a field of positive characteristic $p$ and let the coefficients of the Hilbert series $H(R,t)$ of the finitely generated graded algebra $R$
are bounded by a constant. If $H(R,t)$ is transcendental, then the base field $K$ contains an element
which is not algebraic over the prime subfield ${\mathbb F}_p$ of $K$. For every such field $K$ there exist graded algebras $R$
with transcendental Hilbert series $H(R,t)$ with coefficients bounded by a constant.
\end{theorem}

In the next sections we shall discuss the problem how to construct more algebras with algebraic and nonrational Hilbert series.

By Proposition \ref{algebraic Hilbert series} if the Hilbert series $H(R,t)$ is algebraic,
then its coefficients grow either exponentially or polynomially.
Hence a power series with intermediate  growth of the coefficients
is transcendental. In \cite{G1} Govorov constructed a two-generated monomial algebra
with Hilbert series with intermediate growth of the coefficients.

A very natural class of finitely generated graded algebras
with Hilbert series with coefficients of intermediate growth are universal enveloping algebras of infinite dimensional Lie algebras
of subexponential growth. The first example of this kind was given by Smith \cite{Sm}:

\begin{theorem}\label{example of Martha Smith}
{\rm (i)} If $L$ is an infinite dimensional graded Lie algebra with subexponential growth of the coefficients of its Hilbert series, then
the Hilbert series of its universal enveloping algebra $U(L)$ is with intermediate growth of the coefficients.

{\rm (ii)} There exists a two-generated infinite dimensional graded Lie algebra $L$ with Hilbert series
\[
H(L,t)=t+\frac{1}{1-t}.
\]
Then the Hilbert series of $U(L)$ is with intermediate growth of the coefficients:
\[
H(U(L),t)=\frac{1}{1-t}\prod_{n\geq 1}\frac{1}{1-t^n}.
\]
\end{theorem}

The Lie algebra $L$ in Theorem \ref{example of Martha Smith} (ii) has a basis $\{x,y_1,y_2,\ldots\}$,
$\deg(x)=1$, $\deg(y_i)=i$, $i=1,2,\ldots$, and the defining relations of $L$ are
\[
[x,y_i]=y_{i+1},\quad [y_i,y_j]=0,\quad i,j=1,2,\ldots.
\]

Lichtman \cite{L} generalized the result of Smith for different classes of Lie algebras.
Later Petrogradsky \cite{P1, P2} developed the theory of functions with intermediate growth of the coefficients
which are realized as Hilbert series in the known examples of algebras with intermediate growth.
In this way he introduced a detailed scale to measure the growth of algebras which reflected also on the growth of the coefficients
of the Hilbert series of graded associative and Lie algebras.

The algebras in the examples of Smith \cite{Sm}, Lichtman \cite{L}, and Petrogradsky \cite{P1, P2} are not finitely presented.
Borho and Kraft [BK] conjectured that finitely
presented associative algebras cannot be of intermediate growth.
For a counterexample it is sufficient to show that there exists a finitely presented and
infinite dimensional Lie algebra with polynomial growth.
Leites and Poletaeva \cite{LP} showed that over a field of characteristic 0
the classical Lie algebras $W_d,H_d,S_d,K_d$ of polynomial vector fields are finitely presented.
Recall that the algebra $W_d=\text{Der}(K[X_d])$ consists of the derivations of the polynomial algebra $K[X_d]$.
The special algebra $S_d\subset W_{d+1}$ and the Hamiltonian algebra $H_d\subset W_{2d}$ annihilate suitable exterior differential forms,
and the contact algebra $K_d\subset W_{2d-1}$ multiplies a certain form.
The easiest example is the Witt algebra $W_1$ of the derivations of $K[x]$.

The first example of a finitely presented graded algebra with Hilbert series with intermediate growth of the coefficients
was given by Ufnarovskij \cite{U1}. In his example the algebra is two-generated by elements of degree 1 and 2.
The Lie algebra $W_1$ of the derivations of the polynomial algebra in one variable over a field  $K$ of characteristic 0
has a graded basis
\[
\left\{\delta_{i-1}=x^i\frac{d}{dx}\mid i\geq 0\right\}, \quad \deg\left(x^i\frac{d}{dx}\right)=i-1,
\]
and multiplication
\[
[\delta_{i-1},\delta_{j-1}]=\left[x^i\frac{d}{dx},x^j\frac{d}{dx}\right]
=(j-i)x^{i+j-1}\frac{d}{dx}=(j-i)\delta_{i+j-2}.
\]
Hence for $i\geq 2$ the derivations $\delta_{i+1}$ may be defined inductively by
\[
\delta_{i+1}=\frac{1}{i-1}[\delta_1,\delta_i].
\]

\begin{theorem}\label{example of Ufnarovskij}
Let $L$ be the Lie subalgebra of $W_1$ generated by $\delta_1$ and $\delta_2$. It has a basis
$\{\delta_i\mid i=1,2,\ldots\}$
and defining relations
\[
[\delta_2,\delta_3]=\delta_5\text{ and } [\delta_2,\delta_5]=3\delta_7.
\]
The universal enveloping algebra $U(L)$ of $L$ is generated by $f_1=x$ and $f_2=y$, where
\[
f_{i+1}=\frac{1}{i-1}(f_1f_i-f_if_1),\quad i=2,3,\ldots.
\]
It is a factor algebra of the free algebra $K\langle x,y\rangle$ modulo the ideal generated by
\[
(f_2f_3-f_3f_2)-f_5\text{ and }(f_2f_5-f_5f_2)-3f_7.
\]
If $\deg(f_i)=i$, $i=1,2,\ldots$, then
\[
H(U(L),t)=\prod_{n\geq 1}\frac{1}{1-t^n}.
\]
\end{theorem}

In a note added in the proofs Shearer \cite{Sh} gave two more examples of finitely presented graded algebras with Hilbert series
which also have an intermediate growth of the coefficients. His algebras are generated by three elements and have three defining relations
but, as in the example of Ufnarovskij \cite{U1} one of the generators is of second degree.

\begin{theorem}\label{example of Shearer}
Let $R=K\langle x_1,x_2,y\rangle/(U)$, where
\[
\deg(x_1)=\deg(x_2)=1,\deg(y)=2,
\]
\[
U=\{x_1y-yx_1,x_1x_2x_1-x_2y,x_2^2y\}.
\]
Then the Hilbert series of $R$ is
\[
H(R,t)=\frac{1}{(1-t)(1-t^2)}\prod_{n\geq 1}\frac{1}{1-t^n}.
\]
If in $U$ we replace $x_2^2y$ with $x_2^2$, then
\[
H(R,t)=\frac{1}{(1-t)(1-t^2)}\prod_{n\geq 1}(1+t^n).
\]
\end{theorem}

Ko\c{c}ak \cite{Ko1} modified the construction of Shearer \cite{Sh} such that the three generators are of first degree:

\begin{theorem}\label{three generators of first degree}
Let
\[
U=\{x_2^2x_1-x_1x_2^2,x_2^2x_3-x_1x_3x_1,x_1x_3^2,x_1x_2x_1,x_1x_2x_3,x_3x_2x_1,x_3x_2x_3\}.
\]
Then the coefficients of the Hilbert series of the algebra $R=K\langle x_1,x_2,x_3\rangle/(U)$ are of intermediate growth.
\end{theorem}

Ko\c{c}ak \cite{Ko1} also constructed a graded algebra with quadratic defining relations and intermediate growth of the coefficients of its Hilbert series.

\begin{theorem}\label{quadratic algebra}
Let the Lie algebra $L$ be generated by two elements $x_1$ and $x_2$ of first degree with defining relations
\[
[[[x_1,x_2],x_2],x_2]=[[[x_2,x_1],x_1],x_1]=0,
\]
and let $U(L)=K\oplus U(L)_1\oplus U(L)_2\oplus \cdots$ be its universal enveloping algebra.
Then the coefficients of the Hilbert series of the algebra $R$ are of intermediate growth, where $R$ is generated by
the homogeneous component $U(L)_4$ of degree $4$, and $R$ is a quadratic algebra with $14$ generators and $96$ quadratic relations.
Its growth function $g(n)$ satisfies $g(n)\sim \exp(\sqrt{n})$.
\end{theorem}

The Lie algebra $L$ in Theorem \ref{quadratic algebra} is isomorphic to the Lie algebra of $2\times 2$ matrices with coefficients from $K[z]$
generated by
\[
x_1=\left(\begin{matrix}0&1\\
0&0\\
\end{matrix}\right)\text{ and }
x_2=\left(\begin{matrix}0&0\\
z&0\\
\end{matrix}\right).
\]
The series
\[
\prod_{n\geq 1}\frac{1}{1-t^n}=\sum_{n\geq 0}p_nt^n
\]
and
\[
\prod_{n\geq 1}(1+t^n)=\sum_{n\geq 0}\rho_nt^n
\]
play very special r\^oles in combinatorics: $p_n$ is equal to the number of partitions of $n$
and $\rho_n$ is the number of partitions of $n$ in different parts.
Recall that $\lambda=(\lambda_1,\ldots,\lambda_k)$ is a partition of $n$, if the parts $\lambda_i$ are integers such that
$\lambda_1+\cdots+\lambda_k=n$ and $\lambda_1\geq\cdots\geq\lambda_k\geq 0$; for $\rho_n$ we assume that
$\lambda_1>\cdots>\lambda_k\geq 0$. The asymptotics of $p_n$ and $\rho_n$ was found by
Hardy and Ramanujan \cite{HR} in 1918 and independently by Uspensky \cite{Us} in 1920:
\[
p_n\approx \frac{1}{4n\sqrt{3}}\exp\left(\pi\sqrt{\frac{2}{3}n}\right),\quad
\rho_n\approx \frac{1}{4\sqrt[4]{3n^3}}\exp\left(\pi\sqrt{\frac{1}{3}n}\right).
\]

See also the recent paper by Ko\c{c}ak \cite{Ko2} for more examples and a survey on finitely presented algebras of intermediate growth.

For further reading, including theory of Gr\"obner bases and other combinatorial properties of algebras we refer
e.g. Herzog and Hibi \cite{HH} for commutative algebras
and Ufnarovskij \cite{U3} and Belov, Borisenko, Latyshev \cite{BBL} for noncommutative algebras.

\section{Algebras with prescribed Hilbert series}\label{section prescribed Hilbert series}

In this section we shall discuss the following problem.

\begin{problem}\label{existence of algebras with given Hilbert series}
Given a power series
\[
a(t)=\sum_{n\geq 0}a_nt^n,\quad a_n\in{\mathbb N}_0,
\]
does there exist a finitely generated graded algebra $R$ with Hilbert series equal to $a(t)$
or at least very close to $a(t)$?
\end{problem}

We shall recall the construction of Borho and Kraft \cite{BK} of a finitely generated graded algebra
with Gelfand-Kirillov dimension equal to $\beta\in [2,\infty)$. If $R$ is a finitely generated graded algebra with
$\text{GKdim}(R)=\alpha\in [2,3)$ and $m\in {\mathbb N}$, then the tensor product $K[y_1,\ldots,t_m]\otimes_KR$
is of Gelfand-Kirillov dimension $\alpha+m$. Hence for the construction of an algebra $R$ with $\text{GKdim}(R)\in[2,\infty)$
it is sufficient to handle the case $\text{GKdim}(R)=\alpha\in [2,3)$.
Let $S\subset {\mathbb N}_0$ be a set of nonnegative integers and let
\[
a(t)=\sum_{s\in S}t^s.
\]
We shall construct a two-generated monomial algebra $R$ with Hilbert series
\[
H(R,t)=\frac{1}{1-t}+\frac{t}{(1-t)^2}+\frac{a(t)t^2}{(1-t)^2}.
\]
We fix the set $U\subset \langle x,y\rangle$
\[
U=\{yx^iyx^jy,yx^ky\mid i,j\geq 0,k\in{\mathbb N}_0\setminus S\}.
\]
Then the factor algebra $R=K\langle x,y\rangle/(U)$ of the free algebra $K\langle x,y\rangle$
modulo the ideal generated by $U$ has a basis
\[
\{x^i,x^iyx^j,x^iyx^syx^j\mid i,j\geq 0,s\in S\}
\]
and hence $R$ has the desired Hilbert series.

Pay attention that in the above example the cube $(y)^3$ of the ideal $(y)$ generated by $y$ is equal to zero in $R$.
A similar construction of a two-generated monomial algebra $R$ is given in \cite[Theorem 9.4.11]{D1}.
Assuming that $(y)^k=0$ in $R$, we construct a two-generated monomial algebra $R$ with Hilbert series
\[
H(R,t)=\sum_{i=0}^{k-1}\frac{t^i}{(1-t)^{i+1}}+\frac{a(t)t^k}{(1-t)^k},
\]
A similar approach was used in the recent paper \cite{D2}:

\begin{theorem}\label{main theorem}
Let
\[
a(t)=\sum_{n\geq 0}a_nt^n
\]
be a power series with nonnegative integer coefficients.

{\rm (i)} If $d$ is a positive integer such that $a_n\leq d^n$, $n=0,1,2,\ldots$,
then for any integer $p=0,1,2$, there exists a $(d+1)$-generated monomial algebra $R$ such that its Hilbert series is
\[
H(R,t)=\frac{1}{1-dt}+\frac{t}{(1-dt)^2}+\frac{t^2a(t)}{(1-dt)^p}.
\]

{\rm (ii)} If $\displaystyle a_n\leq \binom{d+n-1}{n-1}$, $n=0,1,2,\ldots$, for some positive integer $d$, then
for any integer $p=0,1,2$, there exists a $(d+1)$-generated graded algebra $R$ such that its Hilbert series is
\[
H(R,t)=\frac{1}{(1-t)^d}+\frac{t}{(1-t)^{2d}}+\frac{t^2a(t)}{(1-t)^{dp}}.
\]
\end{theorem}

Under the assumptions of Theorem \ref{main theorem} (i) a modification of the proof gives
that for any nonnegative integers $p,q$, $p+q\leq 2$,
there exists a $(d+1)$-generated graded algebra $R$ such that its Hilbert series is
\[
H(R,t)=\frac{1}{1-dt}+\frac{t}{(1-dt)^2}+\frac{t^2a(t)}{(1-dt)^p(1-t)^{dq}}.
\]
In the same way we can construct a monomial algebra $R$ with Hilbert series
\[
H(R,t)=\frac{1+2t}{1-dt}-t+t^2a(t).
\]
In all these constructions it is clear that if the power series $a(t)$ is rational, algebraic or transcendental,
the same property has the Hilbert series of the algebra $R$.

\section{PI-algebras}\label{section PI-algebras}

Let $R$ be an algebra  and let $f(x_1,\ldots,x_n)\in K\langle X\rangle=K\langle x_1,x_2,\ldots\rangle$.
We say that $f(x_1,\ldots,x_n)$ is a {\it polynomial identity} for the algebra $R$ if
$f(r_1,\ldots,r_n)=0$ for all $r_1,\ldots,r_n\in R$. If $R$ satisfies a nontrivial polynomial identity
it is called a {\it PI-algebra}.

The study of PI-algebras is an important part of ring theory with a rich structural and combinatorial theory.
PI-algebras form a reasonably big class containing the finite dimensional and the commutative algebras
and enjoying many of their properties.
In this section we shall discuss only the growth and the Hilbert series of finitely generated PI-algebras.
For more details we refer to the survey article \cite{D4}.

One of the main combinatorial theorems for finitely generated PI-algebras is the {\it Shirshov Height Theorem} \cite{S}.

\begin{theorem}\label{Shirshov theorem}
Let $R$ be a PI-algebra generated by $d$ elements $r_1,\ldots,r_d$ and satisfying a polynomial identity of degree $k$.
Then there exists a positive integer $h=h(d,k)$ such that as a vector space $R$ is spanned on the products
$u_1^{n_1}\cdots u_h^{n_h}$, $n_i\geq 0$, $i=1,\ldots,h$, and every $u_i$ is of the form $u_i=r_{j_1}\cdots r_{j_p}$
with $p\leq k-1$.
\end{theorem}

The integer $h$ is called the {\it height} of $R$.

\begin{corollary}\label{polynomial growth for PI-algebras}
Let $R$ be a $d$-generated PI-algebra satisfying a polynomial identity of degree $k$. Then
the growth function of $R$ is bounded by a polynomial of degree $h$ where $h=h(d,k)$ is the height in the theorem of Shirshov.
\end{corollary}

\begin{proof}
Let the algebra $R$ be generated by $r_1,\ldots,r_d$. Then the number of all words $u=r_{j_1}\cdots r_{j_p}$ of length $p$ is equal to $d^p$.
Hence all words of length $\leq k-1$ are $1+d+d^2+\cdots+d^{k-1}$. If we extend the generating set of $R$ to the set of all words of length $\leq k-1$,
Theorem \ref{Shirshov theorem} implies that as a vector space $R$ behaves
as a finite sum of polynomial algebras $K[u_{i_1},\ldots,u_{i_h}]$. Hence the growth function of $R$ is bounded by a polynomial of degree $h$.
\end{proof}

As an immediate consequence we obtain the following theorem of Berele \cite{B}.

\begin{theorem}\label{GKdim of PI-algebras is finite}
Every finitely generated PI-algebra $R$ is of finite Gelfand-Kirillov dimension.
If $R$ is $d$-generated and satisfies a polynomial identity of degree $k$, then $\text{\rm GKdim}(R)\leq h$,
where $h=h(d,k)$ is the height in the Shirshov Height Theorem.
\end{theorem}

The original estimate for the height $h$ in terms of the number of generators $d$ of $R$
and the degree $k$ of the satisfied polynomial identity can be derived from a lemma of Shirshov on combinatorics of words.
There are many attempts to improve the estimates for $h$ and to decrease the length $p\leq k-1$ of the words
$u_i=r_{j_1}\cdots r_{j_p}$ in the Shirshov Height Theorem \ref{Shirshov theorem}.
Shestakov conjectured (see the abstract of the talk of Lvov \cite{Lv}) that the bound $k-1$ for length can be reduced to
$\lceil k/2\rceil$, where , as usually, $\lceil \alpha\rceil$, $\alpha\in \mathbb R$, is the integer part of $\alpha$.
Lvov added some additional arguments which replace $\lceil k/2\rceil$ with the PI-degree $\text{PIdeg}(R)$ of $R$ in the conjecture of Shestakov.
Recall that a PI-algebra $R$ is of {\it PI-degree} $c$ (or of {\it complexity} $c$),
if $c$ is the largest integer such that all
multilinear polynomial identities of $R$ follow from the
multilinear identities of the $c\times c$ matrix algebra $M_c(K)$.
The conjecture of Shestakov was confirmed by
Ufnarovskij \cite{U2}, Belov \cite{Be} and Chekanu \cite{Ch}.
Other proofs are given in the survey article by Belov, Borisenko and Latyshev \cite{BBL}
and in the book by the author and Formanek \cite{DF}.
Concerning the height $h$ the original proof of Shirshov \cite{S} gives primitive recursive estimates.
Later it was shown that $h$ is exponentially bounded
in terms of the number of generators $d$ of the algebra $R$ and the degree $k$ of the polynomial identity,
see the references in the paper by Belov and Kharitonov \cite{BeK}.
In the same paper Belov and Kharitonov found a subexponential bound for $h$: For a fixed $d$ and $k$ sufficiently large
\[
h<k^{12(1+o(1))\log_3k}.
\]

Theorems \ref{Shirshov theorem} and \ref{GKdim of PI-algebras is finite} confirm
that from many points of view finitely generated PI-algebras are similar to commutative algebras.
There are also essential differences. The Gelfand-Kirillov dimension of a finitely generated commutative algebra is an integer.
The discussed in Section \ref{section prescribed Hilbert series} examples
of two-generated PI-algebras $R$ of Gelfand-Kirillov dimension $\alpha\in [2,3)$
and the tensor products $K[y_1,\ldots,t_m]\otimes_KR$ from \cite{BK}
satisfy the polynomial identity
\[
(x_1x_2-x_2x_1)(x_3x_4-x_4x_3)(x_5x_6-x_6x_5)=0.
\]
The examples in \cite[Theorem 9.4.11]{D1} are two-generated and satisfy the polynomial identity
\[
(x_1x_2-x_2x_1)\cdots(x_{2m-1}x_{2m}-x_{2m}x_{2m-1})=0
\]
for a suitable $m$.
Another difference is that the Hilbert series of a finitely generated commutative graded algebra $R$ is rational
and for PI-algebras $R$ it may be also transcendental. In the next section we shall see that for graded PI-algebras
$H(R,t)$ cannot be algebraic and nonrational.

On the other hand, there is an important class of PI-algebras which play the same r\^ole as the polynomial algebras in commutative algebra
and the free associative algebras in the theory of associative algebras.

\begin{definition}
Let $I(R)\subset K\langle X\rangle$ be the ideal of all polynomial identities of the algebra $R$ (such ideals are called {\it T-ideals}).
The factor algebra
\[
F_d(\text{var}R)=K\langle X_d\rangle/(K\langle X_d\rangle\cap I(R))
\]
is called the {\it relatively free algebra of rank $d$ in the variety of algebras $\text{\rm var}R$ generated by} $R$.
\end{definition}

Kemer developed the structure theory of T-ideals in the free algebra $K\langle X\rangle$ over a field $K$ of characteristic 0
in the spirit of classical ideal theory in commutative algebras,
which allowed him to solve several outstanding open problems in the theory of PI-algebras, see \cite{Ke} for an account.
It is well known that over an infinite field $K$ all relatively free algebras are graded and it is a natural question to study their Hilbert series.
Using the results of Kemer, Belov \cite{Be1} established the following theorem which shows that
relatively free algebras share many nice properties typical for commutative algebra.

\begin{theorem}
Let $K$ be a field of characteristic $0$ and let $R$ be a PI-algebra.
Then the Hilbert series $H(F_d(\text{\rm var}R),t)$ is a rational function with denominator similar to the denominators
of the Hilbert series of finitely generated graded commutative algebras.
\end{theorem}

\section{Algebraic and transcendental power series}\label{section power series}

The following partial case of a classical theorem of Fatou \cite{F} from 1906 shows that
the condition that a power series with nonnegative integer coefficients is algebraic is very restrictive.

\begin{theorem}\label{Theorem of Fatou}
If the coefficients of a power series are nonnegative integers and are bounded polynomially,
then the series is either rational or transcendental.
\end{theorem}

The coefficients of the Hilbert series of graded algebras of finite Gelfand-Kirillov dimension grow polynomially.
Hence we obtain immediately the following consequence of Theorem \ref{Theorem of Fatou}.

\begin{theorem}\label{dichotomy of finite GKdim}
The Hilbert series of a finitely generated graded algebra of finite Gelfand-Kirillov dimension is either rational or transcendental.
\end{theorem}

Corollary \ref{polynomial growth for PI-algebras} and Theorem \ref{GKdim of PI-algebras is finite} imply that the same dichotomy holds also for
finitely generated graded PI-algebras.

\begin{theorem}\label{dichotomy Hilbert series of PI-algebras}
The Hilbert series of a finitely generated graded PI-algebra is either rational or transcendental.
\end{theorem}

In order to construct the graded algebras with algebraic or transcendental Hilbert series in Section \ref{section prescribed Hilbert series}
we need algebraic and transcendental power series with nonnegative integer coefficients. We shall survey several methods for construction of
transcendental power series. We already discussed in Section \ref{section growth and Hilbert series}
that the power series with intermediate growth of the coefficients are transcendental.

Recall that the power series $a(t)$ is {\it lacunary}, if
\[
a(t)=\sum_{k\geq 1}a_{n_k}t^{n_k},\quad a_{n_k}\not=0,\quad \lim_{k\to\infty}(n_{k+1}-n_k)=\infty.
\]
Maybe the best known example of such series is
\[
a(t)=\sum_{n\geq 1}t^{n!}
\]
which produces the first explicitly given transcendental number
\[
a\left(\frac{1}{10}\right)=\sum_{n\geq 1}\frac{1}{10^{n!}},
\]
the constant of Liouville \cite{Li}.
The following theorem is due to Mahler \cite[p. 42]{Ma}.
\begin{theorem}\label{theorem of Mahler}
Lacunary series with nonnegative integer coefficients are transcendental.
\end{theorem}

\begin{example}\label{example in Nishioka}
The following power series satisfy the conditions in Theorem \ref{theorem of Mahler}.
A direct proof of their transcendency is given in the book of Nishioka \cite[Theorem 1.1.2]{N}:
\[
a(t)=\sum_{n\geq 0}t^{d^n},\quad d\geq 2.
\]
\end{example}

In the definition of lacunary series we do not restrict the growth of the coefficients although in the above given examples
the nonzero coefficients are equal to 1.
Another way to construct transcendental series with
polynomial or exponential growth of the coefficients uses multiplicative functions, i.e.
functions $\alpha:{\mathbb N}\to {\mathbb N}_0$ satisfying $\alpha(n_1)\alpha(n_2)=\alpha(n_1n_2)$, $n_1,n_2\in\mathbb N$.
S\'ark\"ozy \cite{Sa} described the functions $\alpha$ such that the generating function
$\displaystyle a(t)=\sum_{n\geq 1}a_nt^n$ of the sequence $a_n=\alpha(n)$, $n=1,2,\ldots$, is rational.
Later B\'ezivin \cite{Bez} extended this result for algebraic generating functions.
Recently, another, more number-theoretic proof of the theorem of B\'ezivin was given by Bell, Bruin, and Coons \cite{BBC}.

\begin{theorem}\label{theorem of Bezivin}
Let $\alpha:{\mathbb N}\to {\mathbb N}_0$ be a multiplicative function such that its generating function
\[
a(t)=\sum_{n\geq 1}\alpha(n)t^n
\]
is algebraic. Then either $\alpha(n)=0$ for all sufficiently large $n$, i.e. $a(t)$ is a polynomial,
or there exists a nonnegative integer $k$ and a multiplicative
periodic function $\chi:{\mathbb N}\to {\mathbb Q}$ such that $\alpha(n)=n^k\chi(n)$.
\end{theorem}

The multiplicative periodic functions which appear in Theorem \ref{theorem of Bezivin} of B\'ezivin were described by
Leitmann and Wolke \cite{LW}.

The proof of the following corollary can be found in \cite{BBC}. Here we give simplified arguments.

\begin{corollary}\label{dichotomy for multiplicative functions}
If $\alpha:{\mathbb N}\to {\mathbb N}_0$ is a multiplicative function, then the generating function
\[
a(t)=\sum_{n\geq 1}\alpha(n)t^n
\]
is either rational or transcendental.
\end{corollary}

\begin{proof}
Let the generating function $a(t)$ of the multiplicative function $\alpha:{\mathbb N}\to {\mathbb N}_0$ be algebraic.
By Theorem \ref{theorem of Bezivin} $a(t)$ is either a polynomial (hence a rational function) or $\alpha$ is of the form
$\alpha(n)=n^k\chi(n)$, $n=1,2,\ldots$, where $k\in{\mathbb N}_0$ and $\chi$ is a multiplicative periodic function.
The periodicity of $\chi$ implies that it is bounded. Hence $\alpha(n)\leq n^kc$ for some constant $c>0$
and the coefficients of the power series $a(t)$ grow polynomially. By Theorem \ref{Theorem of Fatou} of Fatou
the power series $a(t)$ cannot be algebraic and nonrational.
\end{proof}

Now it is easy to construct multiplicative functions with transcendental generating function.
The following simple example is from \cite{D2}.

\begin{example}
If $\alpha:{\mathbb N}\to {\mathbb N}_0$ is a multiplicative function it is completely determined by its its values on the prime numbers $p$.
Let $\alpha(p)=q$, where the $q$'s are pairwise different primes and $\alpha(p)\not=p$ for all prime $p$.
If the generating function $\displaystyle a(t)=\sum_{n\geq 1}\alpha(n)t^n$ is rational, then there exists a positive integer $k$
and a periodic multiplicative function $\chi:{\mathbb N}\to {\mathbb Q}$ such that
\[
\alpha(p)=p^k\chi(p)=q,\quad \chi(p)=\frac{q}{p^k}.
\]
Therefore the multiplicative function $\chi$ is not periodic and this implies that $a(t)$ cannot be rational.
\end{example}

By the theorem of Govorov \cite{G1} the Hilbert series $H(R,t)$ of the finitely presented monomial algebra $R$
is a rational function.
Ufnarovskij \cite{U4} gave a construction which associates to $R$ a finite oriented graph $\Gamma(R)$.

\begin{definition}\label{Ufnarovski graph}
Let
\[
R=K\langle X_d\rangle/(U)\quad U\subset \langle X_d\rangle, \vert U\vert<\infty,
\]
be a finitely presented monomial algebra and let $k+1$ be the maximum of the degrees of the monomials in the set $U$.
The following graph $\Gamma(R)$ is called the {\it Ufnarovskij graph}.
The set of the vertices of $\Gamma(R)$ consists of all monomials
of degree $k$ which are not divisible by a monomial in $U$. Two vertices $v_1$ and $v_2$ are connected by an oriented edge from $v_1$ to $v_2$
if and only if there are two elements $x_i,x_j\in X_d$ such that $v_1x_i=x_jv_2\not\in U$. Then the edge is labeled by $x_i$.
(Multiple edges and loops are allowed.) The generating function
\[
g(\Gamma(R),t)=\sum_{n\geq 1}g_nt^n,
\]
of the graph $\Gamma(R)$ has coefficients $g_n$ equal to the number of paths of length $n$.
\end{definition}

The algebra $R$ in the above definition has a basis
consisting of all monomials in $\langle X_d\rangle$ which are not divisible by a monomial in $U$.
The edges of $\Gamma(R)$ are in a bijective correspondence with the basis elements of degree $k+1$ of $R$
and the paths of length $n$ are in bijection with the monomials of degree $n+k$ in the basis.
Ufnarovskij \cite{U4} gave simple arguments
(based on the Cayley-Hamilton theorem only) for the proof of the following result.

\begin{theorem}\label{rationality of graph of Ufnarovksi}
Let $R=K\langle X_d\rangle/(U)$ be a finitely presented monomial algebra and let the maximum of the degrees of the monomials in $U$ is equal to $k+1$.
Then the generating function $g(\Gamma(R),t)$ of the graph $\Gamma(R)$ is a rational function.
The Hilbert series $H(R,t)$ of $R$ and the generating function $g(\Gamma(R),t)$ are related by
\[
H(R,t)=\sum_{n\geq 0}a_nt^n=\sum_{n=1}^ka_nt^n+t^kg(\Gamma(R),t).
\]
Hence $H(R,t)$ is a rational function.
\end{theorem}

Now the theorem of Govorov \cite{G1} is an obvious consequence of Theorem \ref{rationality of graph of Ufnarovksi}.
Additionally, the growth of the finitely presented monomial algebra $R$ can be immediately determined
from purely combinatorial properties of its graph $G(R)$ -- the existence of cycles and their disposition.

The construction of Ufnarovskij can be translated in terms of automata theory and theory of formal languages.

A {\it language $L$ on the alphabet} $X_d$ is a subset of $\langle X_d\rangle$. The language $L$ is {\it regular}
if it is obtained from a finite subset of $\langle X_d\rangle$ applying a finite number of operations of union, multiplication,
and the operation $\ast$ defined by $\displaystyle T^{\ast}=\bigcup_{n\geq 1}T^n$, $T\subset \langle X_d\rangle$.

In the theory of computation a {\it deterministic finite automaton} is a five-tuple $A=(S,X_d,\delta,q_{0},F)$, where
$S$ is a finite set of {\it states}, $X_d$ is a finite alphabet,
$\delta:S\times X_d\to S$ is a {\it transition function}, $s_0$ is the {\it initial} or the {\it start state},
and $F\subseteq S$ is the (possible empty) set of the {\it final states}.
The automaton $A$ can be identified with a {\it finite directed graph} $\Gamma(A)$.
The set of states $S$ is identified with the set of vertices of $\Gamma(A)$.
Each vertex $v\in S$ is an {\it origin} of $d$ edges labeled by the elements of $X_d$
and $v_2$ is the {\it destination} of the edge from $v_1$ to $v_2$ labeled by $x_i$ if $\delta(v_1,x_i)=v_2$.
The language $L(A)$ {\it recognized by the automaton} $A$
consists of all words $x_{i_1}\cdots x_{i_n}$ such that starting from the initial state $s_0$ and following the edges labeled by $x_{i_1},\ldots,x_{i_n}$
we reach a vertex $f$ from the set of final states $F$.
The theorem of Kleene connects deterministic finite automata and regular languages.

\begin{theorem}\label{Theorem of Kleene}
A language $L$ is regular if and only if it is recognized by a deterministic finite automaton.
\end{theorem}

For a background on the topic we refer e.g. to the book by Lallement \cite{La}.

Ufnarovskij \cite{U5} introduced the notion of an automaton monomial algebra.

\begin{definition}\label{automaton algebra}
Let
\[
R=K\langle X_d\rangle/(U), \quad U\subset\langle X_d\rangle,
\]
be a monomial algebra. It is called {\it automaton} if the set of monomials in $\langle X_d\rangle$
not divisible by a monomial from $U$ (which form a basis of $R$) is a regular language. Equivalently,
if $U$ is a minimal set of relations, then $U$ is also a regular language.
\end{definition}

It is known that when $L\subset \langle X_d\rangle$ is a regular language,
then the generating function $g(L,t)$ of the sequence of the numbers of its words of length $n$ is a rational function.
Since finite sets $U\subset\langle X_d\rangle$ are regular languages, this gives one more proof of the theorem of Govorov \cite{G1}.
Involving methods of graph theory Ufnarovskij \cite{U5} showed how to construct a basis of the automaton algebra $R$
and to compute efficiently it growth and Hilbert series.

For further results, see e.g. the paper by M{\aa}nsson and Nordbeck \cite{ManN}
where the authors introduce the generalized Ufnarovskij graph
and as an application show how this construction can be used to test Noetherianity of automaton algebras.
Another application is given by Ced\'o and Okni\'nski \cite{CO} who proved that
every finitely generated algebra which is a finitely generated module of a finitely generated commutative subalgebra is automaton.
See also Ufnarovski \cite{U6} and M\.{a}nsson \cite{Man} for applying computers for explicit calculations.

The above discussions show that it is relatively easy to construct algebras with rational Hilbert series.
It is more difficult to construct algebras with algebraic and nonrational Hilbert series.
Now we shall survey some constructions of algebraic power series using automata theory and theory of formal languages.
Recently there are new applications the theory of regular languages and the theory of finite-state automata
which give new results and new proofs of old results providing algebras with rational and algebraic nonrational Hilbert series,
see La Scala \cite{LS}, La Scala, Piontkovski and Tiwari \cite{LSPT} and La Scala and Piontkovski \cite{LSP} and the references there.

\section{Planar rooted trees and algebraic series}\label{section planar trees}

In this section we shall present another method for construction of algebraic power series with nonnegative integer coefficients.
The leading idea is to start with a sequence of finite sets of objects $A_n$, $n=0,1,2,\ldots$,
for which we know (or can prove), that the generating function
\[
a(z)=\sum_{n\geq 0}\vert A_n\vert z^n
\]
of the sequence $\vert A_n\vert$, $n=0,1,2,\ldots$, is algebraic and nonrational.

A motivating example are the Catalan numbers. The $n$-th Catalan number $c_n$ is equal to the number of planar binary rooted trees with $n$ leaves.
\begin{center}
\includegraphics[width=4cm]{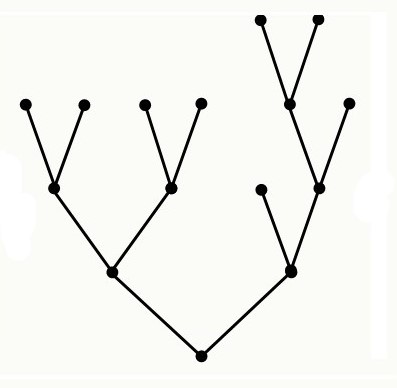}\\
Fig. 1
\end{center}
\medskip
We may introduce the operation {\it gluing of trees} in the set of planar binary rooted trees
which gives it the structure of a nonassociative groupoid
(or a nonassociative {\it magma}):
\begin{center}
\includegraphics[width=10cm]{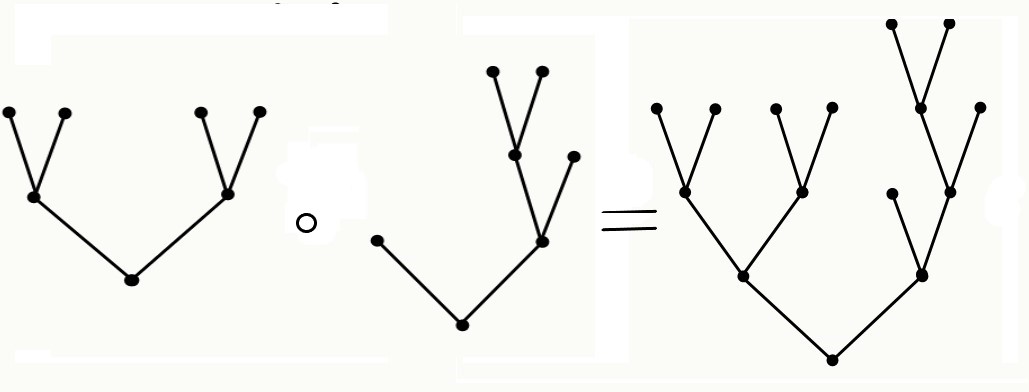}\\
Fig. 2
\end{center}
\medskip
Clearly, this magma is isomorphic to the one-generated free magma $\{x\}$. For example, the tree in Fig. 1 correspond to the nonassociative monomial
$((xx)(xx))(x((xx)x))$ and the gluing of the trees in Fig. 2 can be interpreted as the concatenation of the monomials
$(xx)(xx)$ and $x((xx)x)$ preserving the parentheses:
\[
(xx)(xx)\circ x((xx)x)=((xx)(xx))(x((xx)x)).
\]
Hence, as it is well known, the Catalan numbers satisfy the recurrence relation
\[
c_n=\sum_{k=1}^{n-1}c_kc_{n-k},\quad n=2,3,\ldots,
\]
which implies that their generating function
\[
c(t)=\sum_{n\geq 1}c_nt^n
\]
satisfies the quadratic equation $c^2(t)=c(t)-t$. This also gives the formulas
\[
c(t)=\frac{1-\sqrt{1-4t}}{2},\quad c_n=\frac{1}{n}\binom{2n-2}{n-1},\quad n=1,2,\ldots.
\]
In this way we obtain a nonrational power series which is algebraic.
More generally we may consider the generating function which counts the planar rooted trees
with fixed number of outcoming branches in each vertex, see, e.g. Drensky and Holtkamp \cite{DH}.
This can be formalized in the language of universal algebra in the following way.

We start with a set
\[
\Omega=\Omega_2\cup\Omega_3\cup\cdots
\]
which is a union of finite sets of $n$-ary operations
\[
\Omega_n=\{\nu_{ni}\mid i=1,\ldots,p_n\},\quad n\geq 2,
\]
and an arbitrary set of variables $Y$. We consider the free $\Omega$-magma
\[
\{Y\}_{\Omega}=\mathcal{M}_{\Omega}(Y).
\]
The elements of $\{Y\}_{\Omega}$ are the $\Omega$-monomials which are built inductively.
We assume that $Y\subset \{Y\}_{\Omega}$ and if $u_1,\ldots,u_n\in \{Y\}_{\Omega}$,
then $\nu_{ni}(u_1,\ldots,u_n)$ also belongs to $\{Y\}_{\Omega}$.
In the same way as one constructs the free associative algebra $K\langle Y\rangle$ as the vector space with basis the elements of the
free semigroup $\langle Y\rangle$ and the free nonassociative algebra $\{Y\}$ starting with the free magma $\{Y\}$,
one can construct the free $\Omega$-algebra  $K\{Y\}_{\Omega}$.
This allows to use methods and ideas of ring theory for the study of free $\Omega$-magmas.
The elements of $\{Y\}_{\Omega}$ can be described in terms
of labeled reduced planar rooted trees in a way similar to the way
we identify the free magma $\{x\}$ with the set of planar binary rooted trees.

If $T$ is a planar rooted tree we orient the edges  in direction from the root to the leaves.
We assume that the tree is reduced, i.e. from each vertex which is not a leaf there are at least two outcomming edges.
Then we label the leaves with the elements of $Y$ and if a vertex is with $n$ outcomming edges
we label it with an $n$-ary operation $\nu_{ni}$.
We call such trees $\Omega$-trees with labeled leaves.
There is a one-to-one correspondence between the $\Omega$-monomials and the $\Omega$-trees with labeled leaves.
For example, if $Y=X=\{x_1,x_2,\ldots\}$, then the monomial
\[
\nu_{31}(\nu_{23}(x_1,x_1),x_3,\nu_{32}(x_2,x_1,x_4))
\]
corresponds to the following tree:
\[
 \xymatrix{ {\bullet}_{x_1}\ar@{-}[dr]& {\bullet}_{x_1}\ar@{-}[d]& {\bullet}_{x_3}\ar@{-}[dd]
 & {\bullet}_{x_2}\ar@{-}[d]& {\bullet}_{x_1}\ar@{-}[dl]& {\bullet}_{x_4}\ar@{-}[dll]\\
& {\bullet}_{\nu_{23}}\ar@{-}[dr] & &{\bullet}_{\nu_{32}}\ar@{-}[dl] & & \\
& & {\bullet}_{\nu_{31}}
\\}
\]
\begin{center}
Fig. 3
\end{center}
The set of $\Omega$-trees with labeled leaves inherits the natural grading of the free $\Omega$-magma $\{Y\}_{\Omega}$:
\[
\deg(\nu_{ni}(u_1,\ldots,u_n))=\sum_{k=1}^n\deg(u_k).
\]
The following proposition describes the generating function of the free $\Omega$-magma $\{Y\}_{\Omega}$
and the Hilbert series of the free $\Omega$-algebra $K\{Y\}_{\Omega}$.

\begin{proposition}\label{Hilbert series of free Omega algebra}
Let
\[
p(t)=\sum_{n\geq 2}p_ny^n=\sum_{n\geq 2}\vert\Omega_n\vert t^n
\]
be the generating function of the set of operations $\Omega=\Omega_2\cup\Omega_3\cup\cdots$.

{\rm (i)} When $Y=\{x\}$ consists of one element,
then the generating function of the free $\Omega$-magma $\{x\}_{\Omega}$ (and the Hilbert series of the free $\Omega$-algebra $K\{x\}_{\Omega}$)
\[
g(\{x\}_{\Omega},t)=H(K\{x\}_{\Omega},t)=\sum_{n\geq 1}\vert\Omega_n\vert t^n
\]
is the only solution $z=f(t)$ of the equation $p(z)-z+t=0$ which satisfies the condition $f(0)=0$.

{\rm (ii)} In the general case, if
\[
Y=Y^{(1)}\cup Y^{(2)}\cup\cdots, \text{ where }Y^{(k)}=\{y\in Y\mid \deg(y)=k\},
\]
and
\[
a(t)=\sum_{k\geq 1}\vert Y^{(k)}\vert t^k
\]
is the generating function of the graded set $Y$, then
\[
z=f(t)=g(\{Y\}_{\Omega},t)=H(K\{Y\}_{\Omega},t)
\]
is the solution of the equation
$p(z)-z+a(t)=0$ satisfying the condition $f(0)=0$.
\end{proposition}

The problem when the series $g(\{x\}_{\Omega},t)=H(K\{x\}_{\Omega},t)$ is algebraic and nonrational depending the properties
of the generating function $p(t)$ from Proposition \ref{Hilbert series of free Omega algebra} was studied in the forthcoming paper by
Drensky and Lalov \cite{DL}. As an immediate consequence of Proposition \ref{Hilbert series of free Omega algebra} we obtain:

\begin{corollary}\label{algebraic series for Omega-magma}
If $p(t)$ is a polynomial (with nonnegative integer coefficients), then $g(\{x\}_{\Omega},t)$ is an algebraic nonrational function.
\end{corollary}

Under some mild conditions the same conclusion holds when $p(t)$ is a rational function. The following remark is based on arguments from \cite{DL}.

\begin{remark}\label{algebraic p(t)}
Let the function $p(t)$ from Proposition \ref{Hilbert series of free Omega algebra} be algebraic and let
$b(t,p(t))=0$ for some polynomial $b(t,z)\in {\mathbb Q}[t,z]$. Hence $g(\{x\}_{\Omega},t)$ is equal to the solution $z=f(t)$ of the equation
$b(z,p(z))=b(f(t),p(f(t)))=0$. Since $p(f(t))=f(t)-t$, we obtain that $b(f(t),f(t)-t)=0$. Hence when the function $p(t)$ is algebraic then
this gives an algorithm which has as an input the polynomial equation $b(t,z)=0$ with coefficients in ${\mathbb Q}[t]$
satisfied by $p(t)$ and as an output the polynomial equation $b(z,z-t)=0$,
again with coefficients in ${\mathbb Q}[t]$, satisfied by $g(\{x\}_{\Omega},t)$.
\end{remark}

\begin{remark}\label{the case of X instead of x}
Up till now in this section we start with an algebraic series with nonnegative integer coefficients and obtain an algebraic equation
satisfied by $g(\{x\}_{\Omega},t)$. Then we want to obtain conditions which guarantee that the series $g(\{x\}_{\Omega},t)$ is not rational.
We can apply a similar strategy working with the free $\Omega$-magma $\{Y\}_{\Omega}$ with larger graded generating sets $Y$.
Depending on the properties of the generating function $a(t)$ of the set $Y$ from Proposition \ref{Hilbert series of free Omega algebra} (ii)
we can handle the following three cases:

(1) Both $p(t)$ and $a(t)$ are polynomials in ${\mathbb Q}[t]$.
Then $g(\{Y\}_{\Omega},t)$ is equal to the solution $z=f(t)$ of the equation $p(z)-z+a(t)=0$ with $f(0)=0$.

(2) Let $p(t)\in {\mathbb Q}[t]$ be a polynomial
and let $a(t)$ be algebraic satisfying the polynomial equation $q(t,a(t))=0$, $q(t,z)\in{\mathbb Q}[t,z]$.
Then $g(\{Y\}_{\Omega},t)$ is the solution $z=f(a(t))$ of the equation $p(z)-z+a(t)=0$. Replacing
$a(t)=f(a(t))-p(f(a(t)))$ in $q(t,a(t))=0$ we obtain that $z=f(a(t))$ is a solution of the polynomial equation
$q(t,z-p(z))=0$, and $q(t,z-p(z))\in {\mathbb Q}[t,z]$.

(3) Both $p(t)$ and $a(t)$ are algebraic functions and $b(t,p(t))=q(t,a(t))=0$ for some polynomials $b(t,z),q(t,z)\in{\mathbb Q}[t,z]$.
Applying the arguments in Remark \ref{algebraic p(t)} we obtain that $g(\{x\}_{\Omega},t)=f(t)$ is a solution $u$ of the polynomial equation
$b(u,u-t)=0$. Hence $g(\{Y\}_{\Omega},t)=f(a(t))$ is a solution $u$ of the equation $b(u,u-a(t))=0$. Since $q(t,a(t))=0$,
the polynomial equations $b(u,u-z)=0$ and $q(t,z)=0$ have a common solution $z=a(t)$. Hence the resultant
$r(t,u)=\text{Res}_z(q(t,z),b(u,u-z))$ of the polynomials $q(t,z),b(u,u-z)\in({\mathbb Q}[t,u])[z]$ is equal to 0 which gives
a polynomial equation $r(t,u)=0$ with a solution $u=f(a(t))$.
\end{remark}

A variety of algebraic systems satisfies the {\it Schreier property} if the subsystems of the free systems are also free.
This holds for example for free groups (the {\it Nielsen-Schreier theorem} \cite{Nie, Sch}, two different proofs can be found in \cite{KM, MKS}),
for free Lie algebras (the {\it Shirshov theorem} \cite{S1}), for free nonassociative and free $\Omega$-algebras ({\it theorems of Kurosh} \cite{Ku1, Ku2}).
It is folklorely known that {\it any $\Omega$-submagma of the free $\Omega$-magma $\{Y\}_{\Omega}$ is also free}.
A proof can be found e.g. in Feigelstock \cite{Fe}. (This can be derived also from the theorems of Kurosh \cite{Ku1, Ku2}.)

We shall give an example considered in Drensky and Holtkamp \cite{DH}.
The subset $S$ of the magma $\{x\}$ consisting of all nonassociative monomials of even degree is closed under multiplication and hence forms
a free submagma of $\{x\}$. It is easy to see that the set of free generators of $S$ consists of all monomials of the form $u=u_1u_2$,
where both $u_1$ and $u_2$ are of odd degree. Let
\[
a(t)=\sum_{n\geq 1}a_{2n}t^{2n}
\]
be the generating function of the free generating set of $S$.
The generating function $g(S,t)$ of $S$ is expressed in terms of the generating function of the Catalan numbers
\[
g(S,t)=\sum_{n\geq 2}c_{2n}t^{2n}=\frac{1}{2}(c(t)+c(-t)).
\]
From the equation
\[
g^2(S,t)-g(S,t)+a(t)= c^2(a(t))-c(a(t))+a(t)=0
\]
we obtain that $a(t)$ satisfies the quadratic equation
\[
4a^2(t)-a(t)+t^2=0,
\quad
a(t)=\frac{1}{4}c(4t^2),\quad a_{2n}=4^{n-1}c_n.
\]
Applying the Stirling formula for $n!$ after some calculations we obtain
\[
\frac{a_{2n}}{c_{2n}}\approx \frac{1}{2}\sqrt{\frac{2n-1}{n-1}},\quad \lim_{n\to\infty}\frac{a_{2n}}{c_{2n}}=\frac{\sqrt{2}}{2}\approx 0.707105.
\]
Every monomial $u$ of even degree in $\{x\}$ is a product of two submonomials $u_1$ and $u_2$
where both $u_1$ and $u_2$ are either of even or of odd degree.
The above calculations show that the monomials $u=u_1u_2$ with $u_1$ and $u_2$ of odd degree are much more that those of even degree.
This can be translated in the language of planar binary rooted trees with even number of leaves. Every such tree has two branches
which both are of the same parity of the number of leaves.
\begin{center}
\includegraphics[width=8cm]{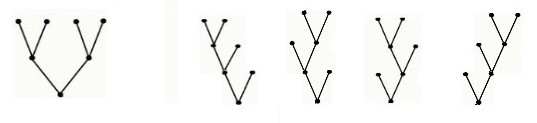}\\
even branches\hskip2.5truecm odd branches\phantom{xxxxxxxxxx}\\
Fig. 4
\end{center}
\medskip
\noindent The trees in Fig. 4 correspond, respectively, to the monomials
\[
(xx)(xx)\text{( even branches), }((xx)x)x,(x(xx))x,x((xx)x),x(x(xx))\text{ (odd branches).}
\]
It turns out that the trees with branches with odd number of leaves are more
than 70 {\%} of all trees with even number of leaves which, at least for the authors of \cite{DH}, was quite surprising.

The above observation was the starting point of the project of Drensky and Lalov \cite{DL}. One of the first results there was the following.

\begin{theorem}\label{submagmas of any Omega-magma}
Let $\Omega$ be a set of operations with algebraic generating function $p(t)$ and let $\{x\}_{\Omega}$ be the one-generated free $\Omega$-magma.
For a fixed positive integer $s$ consider the $\Omega$-submagma $S_{\Omega}$ consisting of all monomials of degree divisible by $s$.
Then the generating function $a(t)$ of the free generating set of $S_{\Omega}$ is algebraic.
\end{theorem}

The following lemma answers the problem when the set $S$ is nonempty.

\begin{lemma}\label{nonempty submagma}
Let the number of the $n$-ary operations in $\Omega$ is equal to $p_n$ and let $d$ be the be the greatest common divisor
of all numbers $n-1$, for which $p_n$ is different from $0$. Then $S_{\Omega}$ is nonempty if and only if $d$ and $s$ are relatively prime.
Moreover, the set $S_{\Omega}$ is either empty or is infinite.
\end{lemma}

One of the main problems in this direction is the following.

\begin{problem}\label{problem for submagmas of any Omega-magma}
If in the notation of Theorem \ref{submagmas of any Omega-magma}
we known the polynomial equation $b(t,z)\in {\mathbb Q}[t,z]$ satisfied by $p(t)$,
how to find the equation satisfied by the generating function $a(t)$ of the free generating set of $S_{\Omega}$?
\end{problem}

In \cite{DL} we have found an algorithm which solves this problem. In particular, we have the following statement which gives more examples
of algebraic power series with nonnegative integer coefficients.

\begin{theorem}
If the generating function  $p(t)$ of the operations in $\Omega$ is a polynomial and the set $S_{\Omega}$ is nonempty,
then the generating function $a(t)$ of the free set of generators of $S_{\Omega}$ is algebraic and nonrational.
\end{theorem}

\begin{example}
Let $\Omega$ consist of one binary operation only and let $s=3$.
This corresponds to the set $S$ of binary planar rooted trees with number of leaves divisible by 3.
Applying the algorithm in \cite{DL} we obtain that the generating function $a(t)$ of the free generating set of $S$ satisfies
\[
729a^4( t)-486a^3( t)+108a^2(t)^2-(64t^3+8)a(t)+16t^3=0.
\]
Solving this equation we obtain four possibilities for $a(t)$.
We expand each of them in series and since only one solution has nonnegative coefficients of the first powers,
we obtain the value of the desired generating function:
\[
a(t) = \frac{1}{6} - \frac{1}{18} \sqrt{1 + 4 t + 16 t^2}
- \frac{\sqrt{ 1 - 2 t - 8 t^2 + \frac{1-64t^3}{\sqrt{1 + 4 t + 16 t^2}} }}{9 \sqrt{2}}
\]
\[
=2 t^3 + 38 t^6 + 1262 t^9 + 51302 t^{12} + 2319176 t^{15} +
 111964106 t^{18} + 5652760340 t^{21}+\cdots.
 \]
\end{example}

\begin{example}\label{super-Catalan numbers mod 2}
Let $\Omega=\Omega_2\cup\Omega_3\cup\cdots$ and let $\vert\Omega\vert=1$ for all $n=2,3,\ldots$.
Its generating function is
\[
p(t)=\frac{t^2}{1-t}.
\]
Then the one-generated free $\Omega$-magma can be identified
with the set of all planar rooted reduced trees and
the generating function of $\{x\}_{\Omega}$
is equal to the generating function of the super-Catalan numbers (see \cite[sequence A001003]{Sl}).
Let $S$ be the set of all monomials of even degree.
The calculations in \cite{DL} give that the generating function $a(t)$ of the set of free generators of $S$ satisfies thee equation
\[
36a^4(t)-12(t^2+1)a^3(t)+ (19t^2+1)a^2(t)+3t^2(t^2-1)a(t)+2t^4=0
\]
Only two of the solutions have nonnegative coefficients of the first few powers.
However, the correct solution is chosen taking into account the coefficient of the 4th power
and it is
\[
a(t))=\frac{1}{12} (1 + t^2) - \frac{1}{12} \sqrt{1 - 34 t^2 + t^4}
- \frac{\sqrt{ t^2 + t^4 - \frac{t^2-34t^4+t^6}{\sqrt{1 - 34 t^2 + t^4}}}}{6 \sqrt{2}}
\]
\[
=t^2 + 10 t^4 + 174 t^6 + 3730 t^8+ 89158 t^{10} + 2278938 t^{12} +
 60962718 t^{14} + 1685358882 t^{16} + \cdots.
\]
\end{example}

\section{Noncommutative invariant theory}\label{section invariant theory}

In this section we shall follow the traditions of classical invariant theory and shall work over the complex field $\mathbb C$
although most of our results are true for any field $K$ is of characteristic 0.
In classical invariant theory one considers the canonical action
of the general linear group $\text{GL}_d({\mathbb C})$ on the $d$-dimensional vector space $V_d$ with basis $\{v_1,\ldots,v_d\}$.
The algebra ${\mathbb C}[X_d]$ consists of the polynomial functions $f(X_d)=f(x_1,\ldots,x_d)$, where
\[
x_i(v)=\xi_i\text{ for }v=\xi_1v_1+\cdots+\xi_dv_d\in V_d,\xi_1,\ldots,\xi_d\in{\mathbb C}.
\]
The group $\text{GL}_d({\mathbb C})$ acts on ${\mathbb C}[X_d]$ by the rule
\[
g(f)(v)=f(g^{-1}(v)),\quad g\in\text{GL}_d({\mathbb C}),f\in{\mathbb C}[X_d],v\in V_d.
\]
If $G$ is a subgroup of $\text{GL}_d({\mathbb C})$, then the algebra ${\mathbb C}[X_d]^G$ of $G$-invariants consists of all
$f(X_d)\in{\mathbb C}[X_d]$ such that
\[
g(f)=f\text{ for all }g\in G.
\]
For a background on classical invariant theory see, e.g. some of the books
by Derksen and Kemper \cite{DeK}, Dolgachev \cite{Do} or Procesi \cite{Pr}.

One possible noncommutative generalization is to replace the polynomial algebra with the free associative algebra ${\mathbb C}\langle X_d\rangle$
under the natural restriction $d\geq 2$.
It is more convenient to assume that $\text{GL}_d({\mathbb C})$ acts canonically on the vector space
${\mathbb C}X_d$ with basis $X_d$ and to extend diagonally its action on ${\mathbb C}\langle X_d\rangle$ by the rule
\[
g(f(x_1,\ldots,x_d))=f(g(x_1),\ldots,g(x_d)),\quad g\in\text{GL}_d({\mathbb C}), f\in{\mathbb C}\langle X_d\rangle.
\]
Then, for a subgroup $G$ of $\text{GL}_d({\mathbb C})$ the algebra of $G$-invariants is
\[
{\mathbb C}\langle X_d\rangle^G=\{f(X_d)\in{\mathbb C}\langle X_d\rangle\mid g(f)=f\text{ for all }g\in G\}.
\]
The algebras of invariants in the commutative case have a lot of nice properties.
For example, the algebra ${\mathbb C}[X_d]^G$ is finitely generated for a large class of groups
including all reductive groups, when $G$  is a maximal unipotent subgroup of a reductive group
(see Had\v{z}iev \cite{Ha} or Grosshans \cite[Theorem 9.4]{Gr}),
and consequently when $G$ is a Borel subgroup of a reductive group.
Since the algebra ${\mathbb C}[X_d]^G$ is graded, the Hilbert-Serre theorem (Theorem \ref{Hilbert series of commutative algebras} (i))
gives that for such groups $G$ the Hilbert series $H({\mathbb C}[X_d]^G,t)$ is a rational function.
In this case the algebra ${\mathbb C}[X_d]^G$
is a homomorphic image of a polynomial algebra ${\mathbb C}[Y_p]$ modulo some ideal $I$.
But it is quite rare when the algebra ${\mathbb C}[X_d]^G$ is isomorphic to the polynomial algebra ${\mathbb C}[Y_p]$.
By the theorem of Shephard and Todd \cite{ShT} and Chevalley \cite{Che} {\it if $G$ is finite then ${\mathbb C}[X_d]^G\cong{\mathbb C}[X_d]$
if and only if $G$ is generated by pseudoreflections.}

The picture of invariant theory for the free algebra ${\mathbb C}\langle X_d\rangle$ is quite different.
The algebra ${\mathbb C}\langle X_d\rangle^G$ is very rarely finitely generated.

\begin{theorem}\label{rarely finitely generated}
{\rm (i) (Dicks and Formanek \cite{DiF} and Kharchenko \cite{Kh2})}
If $G$ is a finite group then ${\mathbb C}\langle X_d\rangle^G$ is finitely generated if and only if
$G$ is cyclic and acts on the vector space ${\mathbb C}X_d$ by scalar multiplication.

{\rm(ii) (Koryukin \cite{Kor})} Let $G$ be an arbitrary subgroup of $\text{\rm GL}_d({\mathbb C})$
and let ${\mathbb C}\langle X_d\rangle^G$ be finitely generated.
Assume that the vector space ${\mathbb C}X_d$ does not have a proper subspace ${\mathbb C}Y_e$, $Y_e=\{y_1,\ldots,y_e\}$, $e<d$,
such that ${\mathbb C}\langle X_d\rangle^G\subseteq {\mathbb C}\langle Y_e\rangle$.
Then $G$ is a finite and acts on ${\mathbb C}X_d$ by scalar multiplication.
\end{theorem}

On the other hand the following theorem of Koryukin \cite{Kor} implies something positive.

\begin{theorem}\label{finite generation by Koryukin}
Let us equip the homogeneous component of degree $n$ of the free algebra ${\mathbb C}\langle X_d\rangle$
with the action of the symmetric group $S_n$ by permuting the positions of the variables:
\[
\left(\sum\alpha_ix_{i_1}\cdots x_{i_n}\right){\sigma}=\sum\alpha_ix_{i_{\sigma(1)}}\cdots x_{i_{\sigma(n)}},\quad \sigma\in S_n.
\]
Then under this additional action the algebra ${\mathbb C}\langle X_d\rangle^G$ is finitely generated for any reductive group $G$.
\end{theorem}

The analogue of the Shephard-Todd-Chevalley theorem  sounds also very different for $K\langle X_d\rangle$.
It turns out that the algebra ${\mathbb C}\langle X_d\rangle^G$ is always free.
Additionally, when $G$ is finite, then there is a Galois correspondence
between the subgroups of $G$ and the free subalgebras of ${\mathbb C}\langle X_d\rangle$ which contain
${\mathbb C}\langle X_d\rangle^G$.

\begin{theorem}\label{Lane-Khanchenko}
{\rm (Lane \cite{Lan} and Kharchenko \cite{Kh1})}
For every subgroup $G$ of $\text{\rm GL}_d({\mathbb C})$
the algebra of invariants ${\mathbb C}\langle X_d\rangle^G$ is free.
\end{theorem}

\begin{theorem}\label{Galois correcpondence} {\rm (Kharchenko \cite{Kh1})}
Let $G$ be a finite subgroup of $\text{\rm GL}_d({\mathbb C})$.
The map $H\longrightarrow {\mathbb C}\langle X_d\rangle^H$ gives a one-to-one
correspondence between the subgroups of $G$
and the free subalgebras of ${\mathbb C}\langle X_d\rangle$ containing
${\mathbb C}\langle X_d\rangle^G$.
\end{theorem}

Comparing with the commutative case,
the behavior of the Hilbert series of ${\mathbb C}\langle X_d\rangle^G$
depends surprisingly very much on the properties of the group $G$.
For example, the classical Molien formula \cite{Mo} for the Hilbert series of the algebra of invariants ${\mathbb C}[X_d]^G$
for a finite group $G$ states that
\[
H({\mathbb C}[X_d]^G,t)=\frac{1}{\vert G\vert}\sum_{g\in G}\frac{1}{\det(1-tg)},
\]
where $\det(1-tg)$ is the determinant of the matrix $I_d-tg\in \text{GL}_d({\mathbb C})$ (and $I_d$ is the identity $d\times d$ matrix).
The analogue of the Molien formula for $H({\mathbb C}\langle X_d\rangle^G,t)$ is due to Dicks and Formanek \cite{DiF}:

\begin{theorem}\label{associative Molien formula}
For a finite subgroup $G$ of $\text{\rm GL}_d({\mathbb C})$
\[
H({\mathbb C}\langle X_d\rangle^G,t)=\frac{1}{\vert G\vert}\sum_{g\in G}\frac{1}{1-\text{\rm tr}(tg)},
\]
where $\text{\rm tr}(tg)$ is the trace of the matrix $tg$, $g\in\text{\rm GL}_d({\mathbb C})$.
\end{theorem}

\begin{corollary}\label{invariants of finite groups assocative}
If $G$ is a finite subgroup of ${\rm GL}_d({\mathbb C})$, then the Hilbert series of the free algebra ${\mathbb C}\langle X_d\rangle^G$
and the generating function $a(t)$ of its set of homogeneous free generators are rational functions.
\end{corollary}

We shall illustrate Corollary \ref{invariants of finite groups assocative} with two examples.

\begin{example}\label{associative symmetric two variables}
Let $d=2$ and $G=S_2$ be the symmetric group of degree 2. It consists of the matrices
\[
I_2=\left(\begin{matrix}1&0\\ 0&1\\ \end{matrix}\right),\quad \sigma=\left(\begin{matrix}0&1\\ 1&0\\ \end{matrix}\right).
\]
Hence
\[
\det(I_2-tI_2)=\left\vert\begin{matrix}1-t&0\\ 0&1-t\\ \end{matrix}\right\vert=(1-t)^2,
\det(I_2-t\sigma)=\left\vert\begin{matrix}1&-t\\ -t&1\\ \end{matrix}\right\vert=1-t^2.
\]
The Molien formula gives
\[
H({\mathbb C}[x_1,x_2]^{S_2},t)=\frac{1}{2}\left(\frac{1}{\det(I_2-tI_2)}+\frac{1}{\det(I_2-t\sigma)}\right)
=\frac{1}{(1-t)(1-t^2)},
\]
which expresses the fact that the algebra ${\mathbb C}[x_1,x_2]^{S_2}$ is isomorphic to the polynomial algebra
generated by the elementary symmetric functions
\[
e_1=x_1+x_2\text{ and }e_2=x_1x_2.
\]
Since $\text{tr}(I_2)=2$, $\text{tr}(\sigma)=0$, by the Dicks-Formanek formula we obtain
\[
H({\mathbb C}\langle x_1,x_2\rangle^{S_2},t)=\frac{1}{2}\left(\frac{1}{1-\text{tr}(tI_2)}+\frac{1}{1-\text{tr}(t\sigma)}\right)
\]
\[
=\frac{1}{2}\left(\frac{1}{1-2t}+1\right)=\frac{1-t}{1-2t}=1+t+2t^2+4t^3+\cdots.
\]
As in the case of free nonassociative algebras, there is a formula for the Hilbert series of the algebra ${\mathbb C}\langle Y\rangle$
for an arbitrary grade set $Y$ of free generators. If the generating function of $Y$ is $a(t)$, then
\[
H({\mathbb C}\langle Y\rangle,t)=\frac{1}{1-a(t)}.
\]
Easy computations give for the free generators of ${\mathbb C}\langle x_1,x_2\rangle^{S_2}$
\[
a(t)=\frac{t}{1-t}.
\]
This shows that the free homogeneous set of generators of the algebra ${\mathbb C}\langle x_1,x_2\rangle^{S_2}$ consists
of one polynomial for each degree $n\geq 1$. This example is a partial case of a result of Wolf \cite{W}
where she studied the symmetric polynomials in the free associative algebra ${\mathbb C}\langle X_d\rangle$, $d\geq 2$.
\end{example}

\begin{example}\label{three variables}
Let $G=\langle\sigma\rangle\subset \text{GL}_3({\mathbb C})$ be the cyclic group of order 3 which permutes the variables $x_1,x_2,x_3$.
It is generated by the matrix
\[
\sigma=\left(\begin{matrix}0&0&1\\ 1&0&0\\ 0&0&1\\ \end{matrix}\right),
\]
\[
\det(I_3-tI_3)=(1-t)^3,\det(I_3-t\sigma)=\det(I_3-t\sigma^2)=1-t^3,
\]
\[
H({\mathbb C}[X_3]^G,t)=\frac{1}{3}\left(\frac{1}{(1-t)^3}+\frac{2}{1-t^3}\right)=\frac{1+t^3}{(1-t)(1-t^2)(1-t^3)}
\]
and this is a confirmation of the well known fact that ${\mathbb C}[X_3]^G$ is a free ${\mathbb C}[e_1,e_2,e_3]$-module generated by
1 and $x_1^2x_2+x_2^2x_3+x_3^2x_1$. Here, as usual,
\[
e_1=x_1+x_2+x_3,\quad e_2=x_1x_2+x_1x_3+x_2x_3,\quad e_3=x_1x_2x_3.
\]
Since $\text{tr}(\sigma)=0$, for $H({\mathbb C}\langle X_3\rangle^G,t)$ we obtain
\[
H({\mathbb C}\langle X_3\rangle^G,t)=\frac{1}{3}\left(\frac{1}{(1-t)^3}+2\right)=\frac{1-2t}{1-3t}=1+t+3t^2+9t^3+\cdots.
\]
For the generating function $a(t)$ of the free generators of ${\mathbb C}\langle X_3\rangle^G$ we have
\[
\frac{1}{1-a(t)}=\frac{1-2t}{1-3t},\quad a(t)=\frac{t}{1-2t}.
\]
\end{example}

The situation changes drastically when we consider arbitrary reductive groups $G$. In the commutative case the Hilbert series
of the algebra of invariants ${\mathbb C}[X_d]^G$ is always rational. Surprisingly even in the simplest noncommutative case
we obtain an algebraic Hilbert series which is not rational.

\begin{example}\label{SL2-invariants free associative algebra}
(i) Let the special linear group $\text{SL}_2=\text{SL}_2({\mathbb C})$ act canonically on the two-dimensional vector space with basis $X_2$.
Almkvist, Dicks and Formanek \cite{ADF} showed that
\[
H({\mathbb C}\langle X_2\rangle^{\text{SL}_2},t)=\frac{1-\sqrt{1-4t^2}}{2t^2}.
\]
This means that homogeneous invariants exist for even degree only and their dimension of degree $2(n-1)$ is equal to the $n$th Catalan number $c_n$.
There is a formula relating the Hilbert series $H({\mathbb C}\langle X_2\rangle^{\text{SL}_2},t)$
and the generating function $a_{\text{SL}_2}(t)$ of the free homogeneous generating set ${\mathbb C}\langle X_2\rangle^{\text{SL}_2}$:
\[
H({\mathbb C}\langle X_2\rangle^{\text{SL}_2},t)=\frac{1}{1-a_{\text{SL}_2}(t)}.
\]
This implies that
\[
a_{\text{SL}_2}(t)=1-\frac{2t^2}{1-\sqrt{1-4t^2}}.
\]

(ii) Drensky and Gupta \cite{DG} computed the Hilbert series of the algebra of invariants ${\mathbb C}\langle X_2\rangle^{\text{UT}_2}$
of the unitriangular group $\text{UT}_2=\text{UT}_2({\mathbb C})$:
\[
H({\mathbb C}\langle X_2\rangle^{\text{UT}_2},t)=\frac{1-\sqrt{1-4t^2}}{t(2t-1+\sqrt{1-4t^2})}.
\]
As in the case of $\text{SL}_2({\mathbb C})$ we can obtain for the free generating set of the algebra of $\text{UT}_2({\mathbb C})$-invariants
\[
a_{\text{UT}_2}(t)=1-\frac{t(2t-1+\sqrt{1-4t^2})}{1-\sqrt{1-4t^2}}=t+a_{\text{SL}_2}(t).
\]
Since ${\mathbb C}\langle X_2\rangle^{\text{SL}_2}\subset{\mathbb C}\langle X_2\rangle^{\text{UT}_2}$, this equality suggests that
the set of free generators of ${\mathbb C}\langle X_2\rangle^{\text{UT}_2}$
consists the free generators of ${\mathbb C}\langle X_2\rangle^{\text{SL}_2}$ and one more generator of first degree
which was confirmed in \cite{DG}. The paper \cite{DG} contains also a procedure
which constructs inductively a free generating set of ${\mathbb C}\langle X_2\rangle^{\text{SL}_2}$.
\end{example}

Invariant theory of $\text{SL}_2({\mathbb C})$ and $\text{UT}_2({\mathbb C})$ considered, respectively,
as subgroups of $\text{SL}_d({\mathbb C})$ and $\text{UT}_d({\mathbb C})$,
acting on the polynomial algebra ${\mathbb C}[X_d]$ and the free associative algebra ${\mathcal C}\langle X_d\rangle$
can be translated in the language of derivations. We shall restrict our considerations for the case $d=2$ only.
Recall that the linear operator $\delta$ acting on an algebra $R$ is called a {\it derivation} if
\[
\delta(r_1r_2)=\delta(r_1)r_2+r_1\delta(r_2)\text{ for all }r_1,r_2\in R.
\]
The derivation is {\it locally nilpotent} if for any $r\in R$ there exists and $n$ such that $\delta^n(r)=0$.
The kernel $R^{\delta}$ of $\delta$ is called its {\it algebra of constants}.
It is well known, see e.g. Bedratyuk \cite{Bed} for comments, references and applications, that
there is a one-to-one correspondence between the ${\mathbb G}_a$-actions (the actions of the additive group $({\mathbb C},+)$)
on ${\mathbb C}X_d$ and the linear locally nilpotent derivations on ${\mathbb C}[X_d]$.

If $\text{UT}_2({\mathbb C})$ acts on ${\mathbb C}[X_2]$ by the rule
\[
g(x_1)=x_1,\quad g(x_2)=x_2+\alpha x_1,\quad g\in\text{UT}_2({\mathbb C}),\alpha\in{\mathbb C},
\]
then ${\mathbb C}[X_2]^{\text{UT}_2}$ coincides with the algebra of constants ${\mathbb C}[X_2]^{\delta_1}$ of the derivation $\delta_1$ defined by
\[
\delta_1(x_1)=0,\quad \delta_1(x_2)=x_1.
\]
Equivalently,
\[
{\mathbb C}[X_2]^{\text{UT}_2}=\{f(x_1,x_2)\in {\mathbb C}[X_2]\mid f(x_1,x_2+x_1)=f(x_1,x_2)\}.
\]
Similarly, ${\mathbb C}[X_2]^{\text{SL}_2}$ coincides with the subalgebra of ${\mathbb C}[X_2]^{\text{UT}_2}$
consisting of all $f(x_1,x_2)\in {\mathbb C}[X_2]^{\text{UT}_2}$ such that
\[
f(x_1+x_2,x_2)=f(x_1,x_2).
\]

Up till now we discussed Hilbert series of algebras of invariants which are subalgebras of polynomial algebras and free associative algebras.
Instead we may consider free algebras in other classes.
One of the most important algebras from this point of view are relatively free algebras of varieties of associative or nonassociative algebras.
We shall restrict our considerations to varieties of associative algebras over $\mathbb C$.

Let $R$ be an associative PI-algebra and let $F_d(\text{var}R)$ be the relatively free algebra of rank $d$
in the variety $\text{\rm var}R$ generated by $R$.
Again, we assume that the general linear group $\text{GL}_d({\mathbb C})$ acts canonically on the vector space ${\mathbb C}X_d$
and extend this action diagonally on the whole algebra $F_d(\text{var}R)$.
(Equipped with this action, in the case when $\text{var}R$ is the variety $\mathfrak A$ of all commutative associative algebras,
we do not consider polynomials as functions. The algebra $F_d({\mathfrak A})$ is isomorphic to the symmetric algebra
$S({\mathbb C}X_d)$ of the vector space ${\mathbb C}X_d$.)
For a subgroup $G$ of $\text{GL}_d({\mathbb C})$ the algebra of $G$-invariants $F_d^G(\text{var}R)$ in defined in an obvious way
as in the case of ${\mathbb C}[X_d]^G$ and ${\mathbb C}\langle X_d\rangle^G$.
For a background on invariant theory of relatively free algebras we refer to the survey articles by Formanek \cite{Fo} and the author \cite{D3},
see also the references in Domokos and Drensky \cite{DomD2, DomD3}.

Although PI-algebras are considered to have many similar properties with commutative algebras,
from the point of view of invariant theory they behave quite different.
For example, the finite generation of $F_d^G(\text{var}R)$ for all finite groups forces very strong restrictions
on the polynomial identities of $R$, and the restrictions are much stronger when we assume that $F_d^G(\text{var}R)$
is finitely generated for all reductive groups, see the surveys \cite{Fo, D3}
and Kharlampovich and Sapir \cite{KhS} where the finite generation is related also with algorithmic problems.
As an illustration we shall mention only a result in Domokos and Drensky \cite{DomD1}.
{\it The algebra $F_d^G(\text{\rm var}R)$ is finitely generated for all reductive groups $G$ if and only $R$ satisfies the identity
of Lie nilpotency $[x_1,\ldots,x_c]=0$ for some $c\geq 2$.}

If we consider Hilbert series of relatively free algebras, they are of the same kind as in the commutative case.
Hence we cannot obtain nonrational algebraic or trascendental power series in this way.
The following theorem was established in Domokos and Drensky \cite{DomD2}.
A key ingredient of its proof is the result of Belov \cite{Be1} for the rationality of the Hilbert series $F_d(\text{\rm var}R)$
and its extension by Berele \cite{B1}.

\begin{theorem}
Let $G$ be a subgroup of $\text{\rm GL}_d({\mathbb C})$ such that
for any finitely generated ${\mathbb N}_0$-graded commutative algebra $A$ with $A_0 = \mathbb C$
on which $\text{\rm GL}_d({\mathbb C})$ acts rationally via graded algebra automorphisms,
the subalgebra $A^G$ of $G$-invariants is finitely generated.
Then for every PI-algebra $R$ the Hilbert series of the relatively free algebra $F_d^G(\text{\rm var}R)$
is a rational function with denominator similar to the denominators of the Hilbert series of the algebras of $G$-invariants
in the commutative case.
\end{theorem}

More applications for computing Hilbert series of invariants of classical groups and important numerical invariants of PI-algebras
can be found in the paper by Benanti, Boumova, Drensky, Genov and Koev \cite{BBDGK}.
Here the usage of derivations is combined with the classical method for solving in nonnegative integers
systems of linear Diophantine equations and inequalities discovered by Elliott \cite{E} from 1903
and its further development by MacMahon \cite{MM} in his ``$\Omega$-Calculus'' or Partition Analysis.

If we go to free nonassociative algebras, the Hilbert series of the algebras of invariants may be even more far from rational than
in the case of free associative algebras.
We shall complete our article with the following result in Drensky and Holtkamp \cite{DH}.

\begin{theorem}
Let ${\mathbb C}\{X_2\}$ be the free two-generated nonassociative algebra. Then the Hilbert series of the algebras of invariants
${\mathbb C}\{X_2\}^{\text{\rm SL}_2}$ and ${\mathbb C}\{X_2\}^{\text{\rm UT}_2}$ are elliptic integrals:
\[
H({\mathbb C}\{X_2\}^{\text{\rm SL}_2},t)=\int_0^1\sin^2(2\pi u)\left(1-\sqrt{1-8t\sin(2\pi u)}\right)du,
\]
\[
H({\mathbb C}\{X_2\}^{\text{\rm UT}_2},t)=\int_0^1\cos^2(\pi u)\left(1-\sqrt{1-8t\cos(2\pi u)}\right)du.
\]
\end{theorem}
The proof uses a noncommutative analogue of the Molien-Weyl integral formula for the Hilbert series in classical invariant theory
(which is an integral version of the Molien formula for finite groups \cite{We1, We2}).

It would be interesting to obtain the Hilbert series for algebras of invariants for the groups $SL_2(K)$ and $UT_2(K)$
acting on other free $\Omega$-algebras, as well for the invariants of other important groups.


\begin{thebibliography}{110}

\bibitem{A}
{\bf N.H. Abel},
{\it {\OE}uvres compl\`etes}, Tome II,
\'Editions J. Gabay, Sceaux, 1992, Reprint of the second edition of 1881.

\bibitem{ADF}
{\bf G. Almkvist, W. Dicks, E. Formanek},
{\it Hilbert series of fixed free algebras and noncommutative classical invariant theory},
J. Algebra {\bf 93} (1985), 189-214.

\bibitem{Ba}
{\bf J. Backelin},
{\it La s\'erie de Poincar\'e-Betti d'une alg\`ebre graduee de type fini \`a une r\'elation est rationnelle},
C. R. Acad. Sci., Paris, S\'er. A {\bf 287} (1978), 843-846.

\bibitem{BD}
{\bf C. Banderier, M. Drmota},
{\it Formulae and asymptotics for coefficients of algebraic functions},
Comb. Probab. Comput. {\bf 24} (2015), No. 1, 1-53.

\bibitem{Bed}
{\bf L. Bedratyuk},
{\it Weitzenb\"ock derivations and classical invariant theory: I. Poincar\'e series},
Serdica Math. J. {\bf 36} (2010), No. 2, 99-120.

\bibitem{BBC}
{\bf J.P. Bell, N. Bruin, M. Coons},
{\it Transcendence of generating functions whose coefficients are multiplicative},
Trans. Am. Math. Soc. {\bf 364} (2012), No. 2, 933-959.

\bibitem{Be}
{\bf A.Ya. Belov},
{\it On a Shirshov basis of relatively free algebras of complexity $n$} (Russian),
Mat. Sb. {\bf 135} (1988), 373-384.
Translation: Math. USSR Sb. {\bf 63} (1988), 363-374.

\bibitem{Be1}
{\bf A.Ya. Belov},
{\it Rationality of Hilbert series of relatively free algebras} (Russian),
Uspekhi Mat. Nauk {\bf 52} (1997), No. 2, 153-154.
Translation: Russian Math. Surveys {\bf 52} (1997), 394-395.

\bibitem{BBL}
{\bf A.Ya. Belov, V.V. Borisenko, V.N. Latyshev},
{\it Monomial algebras},
Algebra, 4, J. Math. Sci. (New York) {\bf 87} (1997), No. 3, 3463-3575.
Russian version: Algebra - 4, Itogi Nauki Tekh., Ser. Sovrem. Probl. Mat. i ee pril. Temat. obz. {\bf 26} (2002), 35-214.

\bibitem{BeK}
{\bf A.Ya. Belov, M.I. Kharitonov},
{\it Subexponential estimates in Shirshov's theorem on height} (Russian),
Mat. Sb. {\bf 203} (2012), No. 4, 81-102.
Translation: Sb. Math. {\bf 203} (2012), No. 4, 534-553.

\bibitem{BBDGK}
{\bf F. Benanti, S. Boumova, V. Drensky, G.K. Genov, P. Koev},
{\it Computing with rational symmetric functions and applications to invariant theory and PI-algebras},
Serdica Math. J. {\bf 38} (2012), Nos 1-3, 137-188.

\bibitem{B}
{\bf A. Berele},
{\it Homogeneous polynomial identities},
Israel J. Math. {\bf 42} (1982), 258-272.

\bibitem{B1}
{\bf A. Berele},
{\it Applications of Belov's theorem to the cocharacter sequence of p.i. algebras},
J. Algebra {\bf 298} (2006), No. 1, 208-214.

\bibitem{Bg}
{\bf G.M. Bergman},
{\it A note on growth functions of algebras and semigroups},
Research Note, Univ. of California, Berkeley, 1978,
unpublished mimeographed notes.

\bibitem{Bez}
{\bf J.-P. B\'ezivin},
{\it Fonctions multiplicatives et \'equations diff\'erentielles},
Bull. Soc. Math. Fr. {\bf 123} (1995), No. 3, 329-349.

\bibitem{BK}
{\bf W. Borho, H.-P. Kraft},
{\it \"Uber die Gelfand-Kirillov-Dimension},
Math. Ann. {\bf 220} (1976), 1-24.

\bibitem{CO}
{\bf F. Ced\'o, J. Okni\'nski},
{\it On a class of automaton algebras},
Proc. Edinb. Math. Soc., II. Ser. {\bf 60} (2017), No. 1, 31-38.

\bibitem{Ch}
{\bf G.P. Chekanu},
{\it On local finiteness of algebras} (Russian),
Mat. Issled. {\bf 105} (1988), 153-171.

\bibitem{Che}
{\bf C. Chevalley},
{\it Invariants of finite groups generated by reflections},
Amer. J. Math. {\bf 77} (1955), 778-782.

\bibitem{Co}
{\bf J. Cockle},
{\it On transcendental and algebraic solution},
Philosophical Magazine {\bf XXI} (1861), 379-383.

\bibitem{C1}
{\bf L. Comtet},
{\it Calcul pratique des coefficients de Taylor d'une fonction alg\'ebrique},
Enseignement Math. (2) {\bf 10} (1964), 267-270.

\bibitem{C2}
{\bf L. Comtet}, {\it Advanced Combinatorics},
D. Reidel Publishing Co., Dordrecht, 1974,
enlarged edition of the 2-volumes {\it Analyse combinatoire}, published in French in 1970,
by Presses Universitaires de France.

\bibitem{DeK}
{\bf H. Derksen, G. Kemper},
{\it Computational Invariant Theory},
Encyclopedia of Mathematical Sciences, {\bf 130}
Springer-Verlag, Berlin, 2002.

\bibitem{DiF}
{\bf W. Dicks, E. Formanek},
{\it Poincar\'e series and a problem of S. Montgomery},
Lin. Multilin. Algebra {\bf 12} (1982), 21-30.

\bibitem{Do}
{\bf I. Dolgachev},
{\it Lectures on Invariant Theory},
London Math. Soc. Lecture Note Series {\bf 296},
Cambridge University Press, 2003.

\bibitem{Dom}
{\bf M. Domokos},
{\it Relatively free invariant algebras of finite reflection groups},
Trans. Amer. Math. Soc.  {\bf 348} (1996), No. 6, 2217-2234.

\bibitem{DomD1}
{\bf M. Domokos, V. Drensky},
{\it A Hilbert-Nagata theorem in noncommutative invariant theory},
Trans. Amer. Math. Soc. {\bf 350} (1998), 2797-2811.

\bibitem{DomD2}
{\bf M. Domokos, V. Drensky},
{\it Rationality of Hilbert series in noncommutative invariant theory},
Internat. J. Algebra Comput. {\bf 27} (2017), No. 7, 831-848.

\bibitem{DomD3}
{\bf M. Domokos, V. Drensky},
{\it Constructive noncommutative invariant theory},
Transformation Groups (to appear), arXiv:1811.06342 [math.RT].

\bibitem{D3}
{\bf V. Drensky},
{\it Commutative and noncommutative invariant theory},
in Math. and Education in Math.,
Proc. of the 24-th Spring Conf. of the Union of Bulgar. Mathematicians,
Svishtov, April 4-7, 1995, Sofia, 1995, 14-50.

\bibitem{D4}
{\bf V. Drensky},
{\it Gelfand-Kirillov dimension of PI-algebras},
in ``Methods in Ring Theory, Proc. of the Trento Conf.'',
Lect. Notes in Pure and Appl. Math. {\bf 198}, Dekker, 1998, 97-113.

\bibitem{D1}
{\bf V. Drensky},
{\it Free Algebras and PI-Algebras. Graduate Course in Algebra},
Springer-Verlag, Singapore, 2000.

\bibitem{D2}
{\bf V. Drensky},
{\it Graded algebras with prescribed Hilbert series},
arXiv:2001.01064v1 [math.RA].

\bibitem{DF}
{\bf V. Drensky, E. Formanek},
{\it Polynomial Identity Rings},
Advanced Courses in Mathematics, CRM Barcelona,
Birkh\"auser, Basel-Boston, 2004.

\bibitem{DG}
{\bf V. Drensky, C.K. Gupta},
{\it Constants of Weitzenb\"ock derivations and invariants of unipotent
transformations acting on relatively free algebras},
J. Algebra {\bf 292} (2005), 393-428.

\bibitem{DH}
{\bf V. Drensky, R. Holtkamp},
{\it Planar trees, free nonassociative algebras, invariants, and elliptic integrals},
Algebra and Discrete Mathematics (2008), No. 2, 1-41.

\bibitem{DL}
{\bf V. Drensky, Ch. Lalov},
{\it Free magmas, planar trees, and their generating functions},
in preparation.

\bibitem{E}
{\bf E.B. Elliott},
{\it On linear homogeneous diophantine equations},
Quart. J. Pure Appl. Math. {\bf 34} (1903), 348-377.

\bibitem{F}
{\bf P. Fatou},
{\it S\'eries trigonom\'etriques et s\'eries de Taylor},
Acta Math. {\bf 30} (1906), 335-400.

\bibitem{Fe}
{\bf S. Feigelstock},
{\it A universal subalgebra theorem},
Amer. Math. Monthly {\bf 72} (1965), 884-888.

\bibitem{Fo}
{\bf E. Formanek},
{\it Noncommutative invariant theory},
Contemp. Math. {\bf 43} (1985), 87-119.

\bibitem{G1}
{\bf V.E. Govorov},
{\it Graded algebras} (Russian),
Mat. Zametki {\bf 12} (1972), No. 2, 197-204.
Translation: Math. Notes {\bf 12} (1972), 552-556.

\bibitem{G2}
{\bf V.E. Govorov},
{\it On the dimension of graded algebras} (Russian),
Mat. Zametki {\bf 14} (1973), No. 2, 209-216.
Translation: Math. Notes {\bf 14} (1973), 678-682.

\bibitem{Gr}
{\bf F.D. Grosshans},
{\it Algebraic Homogeneous Spaces and Invariant Theory},
Lecture Notes in Mathematics, {\bf 1673}, Springer-Verlag, Berlin, 1997.

\bibitem{Ha}
{\bf D\v{z}. Had\v{z}iev},
{\it Some questions in the theory of vector invariants} (Russian),
Mat. Sb. (NS) {\bf 72}({\bf 114}) (1967), 420-435.
Translation: Math. USSR, Sb. {\bf 1} (1967), 383-396.

\bibitem{HR}
{\bf G.H. Hardy, S. Ramanujan},
{\it Asymptotic formulae in combinatory analysis},
Proc. Lond. Math. Soc. (2) {\bf 17} (1918), 75-115.

\bibitem{KhS}
{\bf O.G. Kharlampovich, M.V. Sapir},
{\it Algorithmic problems in varieties},
Internat. J. Algebra Comput. {\bf 5} (1995), Nos. 4-5, 379-602.

\bibitem{H}
{\bf R. Harley},
{\it On the theory of the transcendental solution of algebraic equations},
Quart. Journal of Pure and Applied Math. {\bf 5} (1862), 337-361.

\bibitem{HH}
{\bf J. Herzog, T. Hibi},
{\it Monomial Ideals},
Graduate Texts in Mathematics {\bf 260}, Springer-Verlag, London, 2011.

\bibitem{KBR}
{\bf A. Kanel-Belov, L.H. Rowen},
{\it Computational Aspects of Polynomial Identities},
Research Notes in Mathematics {\bf 9}, A K Peters, Wellesley, MA, 2005.

\bibitem{KM}
{\bf M.I. Kargapolov, Yu.I. Merzlyakov},
{\it Fundamentals of the Theory of Groups} (Russian),
``Nauka'', Moscow, Third edition, 1982, Fourth Edition, 1996.
Translation from the Second Edition:
Grad. Texts in Math. {\bf 62}, Springer-Verlag, New York-Berlin, 1979.

\bibitem{Ke}
{\bf A.R. Kemer},
{\it Ideals of Identities of Associative Algebras},
Translations of Math. Monographs {\bf 87}, American Mathematical Society (AMS),  Providence, RI, 1991.

\bibitem{Kh1}
{\bf V.K. Kharchenko},
{\it Algebra of invariants of free algebras} (Russian),
Algebra i Logika {\bf 17} (1978), 478-487.
Translation: Algebra and Logic {\bf 17} (1978), 316-321.

\bibitem{Kh2}
{\bf V.K. Kharchenko},
{\it Noncommutative invariants of finite groups and Noetherian varieties},
J. Pure Appl. Algebra {\bf 31} (1984), 83-90.

\bibitem{K}
{\bf Y. Kobayashi},
{\it Another graded algebra with a nonrational Hilbert series},
Proc. Am. Math. Soc. {\bf 81} (1981), 19-22.

\bibitem{Ko1}
{\bf D. Ko\c{c}ak},
{\it Finitely presented quadratic algebras of intermediate growth},
Algebra Discrete Math. {\bf 20} (2015), No. 1, 69-88.

\bibitem{Ko2}
{\bf D. Ko\c{c}ak},
{\it Intermediate growth in finitely presented algebras},
Int. J. Algebra Comput. {\bf 27} (2017), No. 4, 391-401.

\bibitem{Kor}
{\bf A.N. Koryukin},
{\it Noncommutative invariants of reductive groups} (Russian),
Algebra i Logika {\bf 23} (1984), No. 4, 419-429.
Translation: Algebra Logic {\bf 23} (1984), 290-296.

\bibitem{KL}
{\bf G.R. Krause, T.H. Lenagan},
{\it Growth of Algebras and Gelfand-Kirillov Dimension}, Revised ed.
Graduate Studies in Mathematics. {\bf 22}, American Mathematical Society (AMS), Providence, RI, 2000.

\bibitem{Ku1}
{\bf A.G. Kurosh},
{\it Non-associative free algebras and free products of algebras} (Russian),
Mat. Sbornik, {\bf 20} (1947), 239-260 (English summary, 260-262).

\bibitem{Ku2}
{\bf A.G. Kurosh},
{\it Free sums of multiple operator algebras} (Russian),
Sib. Mat. Zh., {\bf 1} (1960),  62-70.

\bibitem{La}
{\bf G. Lallement},
{\it Semigroups and Combinatorial Applications},
Pure and Applied Mathematics,
A Wiley-Interscience Publication, John Wiley \& Sons,
New York-Chichester-Brisbane, 1979.

\bibitem{Lan}
{\bf D.R. Lane},
{\it Free Algebras of Rank Two and Their Automorphisms},
Ph.D. Thesis, Bedford College, London, 1976.

\bibitem{LS}
{\bf R. La Scala},
{\it Monomial right ideals and the Hilbert series of noncommutative modules},
J. Symb. Comput. {\bf 80} (2017), Part 2, 403-415.

\bibitem{LSP}
{\bf La Scala, D. Piontkovski},
{\it Context-free languages and associative algebras with algebraic Hilbert series},
arXiv:2001.09112 [math.RA].

\bibitem{LSPT}
{\bf La Scala, D. Piontkovski, S.K. Tiwari},
{\it Noncommutative algebras, context-free grammars and algebraic Hilbert series},
J. Symb. Comput. (to appear).

\bibitem{LP}
{\bf D. Leites, E. Poletaeva},
{\it Defining relations for classical Lie algebras of polynomial vector fields},
Math. Scand. {\bf 81} (1997), No. 1, 5-19.

\bibitem{LW}
{\bf D. Leitmann, D. Wolke},
{\it Periodische und multiplikative zahlentheoretische Funktionen},
Monatsh. Math. {\bf 81} (1976), 279-289.

\bibitem{L}
{\bf A.I. Lichtman},
{\it Growth in enveloping algebras}
Isr. J. Math. {\bf 47} (1984), 296-304.

\bibitem{Li}
{\bf J. Liouville},
{\it Sur des classes tr\`es \'etendues de quantit\'es dont valeur n'est ni alg\'ebrique,
ni m\^eme r\'educible \`a des irrationelles alg\'ebriques},
C.R. Acad. Sci., Paris, S\'er. A {\bf 18} (1844), 883-885.
J. Math. Pures Appl. {\bf 16} (1851), 133-142.

\bibitem{Lv}
{\bf I.V. Lvov},
{\it On the Shirshov height theorem} (Russian),
in ``Fifth All-Union Symposium on the Theory of Rings, Algebras and Modules
(Held at Novosibirsk, September, 21st-23rd, 1982)'',
Abstracts of communications,
Institute of Mathematics SO AN SSSR, Novosibirsk 1982, 89-90.

\bibitem{M}
{\bf F.S. Macaulay},
{\it Some properties of enumeration in the theory of modular systems},
Proc. Lond. Math. Soc. (2) {\bf 26} (1927), 531-555.

\bibitem{MM}
{\bf P.A. MacMahon},
{\it Combinatory Analysis}, vols. 1 and 2,
Cambridge Univ. Press. 1915, 1916.
Reprinted in one volume: Chelsea, New York, 1960.

\bibitem{Ma}
{\bf K. Mahler},
{\it Lectures on Transcendental Numbers},
Edited and completed by B. Divi\v{s} and W.J. Le Veque,
Lecture Notes in Mathematics {\bf 546}, Springer-Verlag, Berlin-Heidelberg-New York, 1976.

\bibitem{MKS}
{\bf W. Magnus, A. Karrass, D. Solitar},
{\it Combinatorial Group Theory},
Interscience, John Wiley and Sons, New York-London-Sydney, 1966.

\bibitem{Man}
{\bf J. M\.{a}nsson},
{\it On the computation of Hilbert series and Poincar\'e series for algebras with infinite Gr\"obner bases},
Comput. Sci. J. Mold. {\bf 8} (2000), No. 1, 42-63.

\bibitem{ManN}
{\bf J. M{\aa}nsson, P. Nordbeck},
{\it A generalized Ufnarovski graph},
Appl. Algebra Eng. Commun. Comput. {\bf 16} (2005), No. 5, 293-306.

\bibitem{Mo}
{\bf T. Molien},
{\it \"Uber die Invarianten der linearen Substitutionsgruppen},
Sitz. K\"onig Preuss. Akad. Wiss. (1897), No. 52, 1152-1156.

\bibitem{Nie}
{\bf J. Nielsen},
{\it Die Isomorphismengruppe der freien Gruppen},
Math. Ann. {\bf 91} (1924), 169-209.

\bibitem{N}
{\bf K. Nishioka},
{\it Mahler Functions and Transcendence},
Lecture Notes in Mathematics {\bf 1631}, Springer-Verlag, Berlin, 1996.

\bibitem{P1}
{\bf V.M. Petrogradsky},
{\it Intermediate growth in Lie algebras and their enveloping algebras},
J.Algebra {\bf 179} (1996), 459-482.

\bibitem{P2}
{\bf V.M. Petrogradsky},
{\it Growth of finitely generated polynilpotent Lie algebras and
groups, generalized partitions, and functions analytic in the unit circle},
Internat. J. Algebra Comput. {\bf 9} (1999), No. 2, 179-212.

\bibitem{Pi}
{\bf D. Piontkovski},
{\it Algebras of linear growth and the dynamical Mordell-Lang conjecture},
Adv. Math. {\bf 343} (2019), 141-156.

\bibitem{Pr}
{\bf C. Procesi},
{\it Lie Groups (An Approach through Invariants and Representations)},
Springer, New York, 2007.

\bibitem{Sa}
{\bf A. S\'ark\"ozy},
{\it On multiplicative arithmetic functions satisfying a linear recursion},
Stud. Sci. Math. Hung. {\bf 13} (1978), 79-104.

\bibitem{Sch}
{\bf O. Schreier},
{\it Die Untergruppen der freien Gruppen},
Abh. Math. Semin. Univ. Hambg. {\bf 5} (1927), 161-183.

\bibitem{Sh}
{\bf J.B. Shearer},
{\it A graded algebra with a non-rational Hilbert series},
J. Algebra {\bf 62} (1980), 228-231.

\bibitem{ShT}
{\bf G.C. Shephard, J.A. Todd},
{\it Finite unitary reflection groups},
Canad. J. Math. {\bf 6} (1954), 274-304.

\bibitem{S1}
{\bf A.I. Shirshov},
{\it Subalgebras of free Lie algebras} (Russian),
Mat. Sb. {\bf 33} (1953), 441-452.

\bibitem{S}
{\bf A.I. Shirshov},
{\it On rings with identity relations} (Russian),
Mat. Sb. {\bf 41} (1957), 277-283.

\bibitem{Sl}
{\bf N. Sloane} (Edt.),
{\it The On-Line Encyclopedia of Integer Sequences}, founded in 1964, https://oeis.org/

\bibitem{Sm}
{\bf M.K. Smith},
{\it Universal enveloping algebras with subexponential but not polynomially bounded growth}
Proc. Am. Math. Soc. {\bf 60} (1976), 22-24.

\bibitem{T}
{\bf J. Tannery},
{\it Propri\'et\'es des int\'egrales des \'equations differentielles lin\'eaires \`a coefficients variables},
Th\`ese de doctorat \'es sciences math\'ematiques, Facult\'e des Sciences de Paris, 1874.
Available at http://gallica.bnf.fr.

\bibitem{U1}
{\bf V.A. Ufnarovskij},
{\it Poincar\'e series of graded algebras} (Russian),
Mat. Zametki {\bf 27} (1980), 21-32.
Translation: Math. Notes {\bf 27} (1980), 12-18.

\bibitem{U4}
{\bf V.A. Ufnarovskij},
{\it A growth criterion for graphs and algebras defined by words} (Russian),
Mat. Zametki {\bf 31} (1982), 465-472.
Translation: Math. Notes {\bf 31} (1982), 238-241.

\bibitem{U2}
{\bf V.A. Ufnarovskij},
{\it An independence theorem and its consequences} (Russian),
Mat. Sb. {\bf 128} (1985), 124-132.
Translation: Math. USSR Sb. {\bf 56} (1985), 121-129.

\bibitem{U5}
{\bf V.A. Ufnarovskij},
{\it On the use of graphs for calculating basis, growth, and Hilbert series of associative algebras} (Russian),
Mat. Sb. {\bf 180} (1989), No. 11, 1548-1560.
Translation: Math. USSR, Sb. {\bf 68} (1989), No. 2, 417-428.

\bibitem{U3}
{\bf V.A. Ufnarovskij},
{\it Combinatorial and asymptotic methods in algebra} (Russian),
Itogi Nauki Tekh., Ser. Sovrem. Probl. Mat., Fundam. Napravleniya {\bf 57} (1990), 5-177.
Translation: in A.I. Kostrikin, I.R. Shafarevich (Eds.),
``Algebra VI'', Encyclopaedia of Mathematical Sciences {\bf 57}, Springer-Verlag, Berlin, 1995, 1-196.

\bibitem{U6}
{\bf V.A. Ufnarovski},
{\it Calculations of growth and Hilbert series by computer},
in Fischer, Klaus G. (ed.) et al., Computational Algebra. Papers from the Mid-Atlantic Algebra Conference, Held at George Mason University,
Fairfax, VA, USA, May 20-23, 1993, Lect. Notes Pure Appl. Math. {\bf 151}, Dekker, . New York, 1993, 247-256.

\bibitem{Us}
{\bf J.V. Uspensky},
{\it Asymptotic formulae for numerical functions which occur in the theory of partitions} (Russian),
Bull. Acad. Sci. URSS {\bf 14} (1920), 199-218.

\bibitem{We1}
{\bf H. Weyl},
{\it Zur Darstellungstheorie und Invariantenabz\"ahlung der projektiven,
der Kom\-p\-lex- und der Drehungsgruppe},
Acta Math. {\bf 48} (1926), 255-278;
reprinted in ``Gesammelte Abhandlungen'', Band III,
Springer-Verlag, Berlin -- Heidelberg -- New York, 1968, 1-25.

\bibitem{We2}
{\bf H. Weyl},
{\it The Classical Groups, Their Invariants and Representations},
Princeton Univ. Press, Princeton, N.J., 1946, New Edition, 1997.

\bibitem{W}
{\bf M.C. Wolf},
{\it Symmetric functions of non-commutative elements},
Duke Math. J. {\bf 2} (1936), No. 4, 626-637.

\end{thebibliography}
\end{document}